\documentclass[10pt,reqno]{amsart}
\usepackage{amssymb,amsmath,amsthm}
\usepackage{esint}
\usepackage{mathrsfs}
\usepackage{mathtools}
\usepackage{comment}
\usepackage{bbm,dsfont}
\usepackage{graphicx}
\usepackage[utf8]{inputenc}
\usepackage[T1]{fontenc}
\usepackage[hidelinks]{hyperref}
\usepackage[top=1.5in, bottom=1.5in, left=1.25in, right=1.25in]	{geometry}
\usepackage{titlesec}
\titleformat{\section}[hang]{\normalfont\scshape\centering}{\thesection.}{1em}{}
\titleformat{\subsection}[runin]{\bf}{\thesubsection.}{1em}{}
\usepackage[english]{babel}
\usepackage{natbib,hyperref}

\newtheorem{theorem}{Theorem}[section]

\newtheorem{lemma}[theorem]{Lemma}
\newtheorem{rmk}[theorem]{Remark}
\newtheorem{definition}[theorem]{Definition}
\newtheorem{proposition}[theorem]{Proposition}
\usepackage[normalem]{ulem}
\usepackage{xcolor}
\usepackage{enumitem} 

\allowdisplaybreaks

\numberwithin{equation}{section}
	
\title[An optimal fractional Hardy inequality on the discrete half-line]{An optimal fractional Hardy inequality\\ on the discrete half-line}
\author{Ujjal Das} 
\address{Ujjal Das 
(\textnormal{{udas@bcamath.org, getujjaldas@gmail.com}})
\newline BCAM -- Basque Center for Applied Mathematics, 48009 Bilbao, Spain}	

\author{Rub\'en de la Fuente-Fern\'andez} 
\address{Rub\'en de la Fuente-Fern\'andez 
(\textnormal{rdelafuente@bcamath.org})
\newline BCAM -- Basque Center for Applied Mathematics, 48009 Bilbao, Spain
\newline Universidad del Pa{\'i}s Vasco / Euskal Herriko Unibertsitatea, 48080 Leioa, Spain}

\thanks{Ujjal Das and Rub\'en de la Fuente-Fern\'andez are partially supported by the Basque Government through the BERC 2022-2025 program and by the Spanish Ministry of Science and Innovation: BCAM Severo Ochoa accreditation CEX2021-001142-S/MICIN/AEI/\allowbreak10.13039/501100011033 and CNS2023-143893 funded by MICIU/AEI/10.13039/501100011033 and by the European Union NextGenerationEU/PRTR. Rub\'en de la Fuente-Fern\'andez is also partially supported by PID2023-146646NB-I00 funded by MICIU/\allowbreak AEI/10.13039/501100011033 and by ESF+.}

\date{\today}

\begin{document}

\begin{abstract}
\noindent
In the context of Hardy inequalities for the fractional Laplacian $(-\Delta_{\mathbb{N}})^{\sigma}$ on the discrete half-line $\mathbb{N}$, we provide an optimal Hardy-weight $W^{\mathrm{op}}_{\sigma}$ for exponents $\sigma\in\left(0,1\right]$. As a consequence, we provide the sharp constant in the fractional Hardy inequality with the classical Hardy-weight $n^{-2\sigma}$ on $\mathbb{N}$.  It turns out that for $\sigma =1$ the Hardy-weight  $W^{\mathrm{op}}_{1}$ is pointwise larger than the optimal Hardy-weight obtained by Keller--Pinchover--Pogorzelski near infinity. 
As an application of our main result, we obtain unique continuation results at infinity for the solutions of some fractional Schr\"odinger equation. 
\end{abstract}

\maketitle

\section{Introduction}

Although the history of Hardy inequalities begins in 1906, it was not until 1920 that a first version of his celebrated theorem was published in \cite{Har20}. Since then, proving this kind of inequality has been an active field of study. For instance, see \cite{FKP, KPPAMS, DAFR, RS_Bull} for improved Hardy-type inequalities, \cite{GKS_Imp,G_CVPDE,G_JF,KPP_PLMS} for Hardy-type inequalities involving polyharmonic operators,  \cite{CR18,GKS25,KN23} for fractional versions and the references therein. We also refer to \cite{KMP06} for a history of Hardy inequalities and \cite{Bal} for recent developments on it. A general form of the Hardy inequality for the standard discrete Laplacian $-\Delta_{\mathbb{N}}$ (defined in \eqref{Discrete-Laplacian}) can be written as
\begin{equation}\label{og_Hardy}
    \langle -\Delta_{\mathbb{N}}  f,f\rangle_{\ell^2(\mathbb{N})}\geq\langle Wf,f\rangle_{\ell^2(\mathbb{N})} \,,
\end{equation}
where $f\in C_ c(\mathbb{N})$, $f(0)=0$ and $W:\mathbb{N} \rightarrow [0,\infty)$ is a non-trivial weight function. The classical Hardy inequality was given with the weight $W(n)=W^{\operatorname{H}}(n):=\frac{1}{4n^2}$, with $1/4$ being the best constant, which means that one can find a function $f \in C_ c(\mathbb{N})$ such that \eqref{og_Hardy} does not hold with $W(n)=(1+\lambda)\frac{1}{4n^2}$ for $\lambda>0$. Nevertheless, almost after a century Keller--Pinchover--Pogorzelski \cite[Theorem 7.3]{KPP18} found that one can substitute $W^{\operatorname{H}}$ with a larger weight, which behaves asymptotically as $\frac{1}{4n^2}$, and \eqref{og_Hardy} still holds true, see also \cite[Corollary 12]{KLS}. Precisely,  their Hardy-weight $W$ is given by
\begin{align} \label{HW-Op_KPP}
   W(n)= W^{\mathrm{KPP}}(n) := 2- \left[\sqrt{1+\frac{1}{n}} + \sqrt{1-\frac{1}{n}} \right] = \frac{1}{4n^2} + \frac{5}{64n^4} + \ldots \,.
\end{align}
see Section \ref{Sec:Final-Rem} for further discussion. Furthermore, the authors prove that $W^{\mathrm{KPP}}$ is {\it optimal} in the sense that \eqref{og_Hardy} fails for any pointwise larger weight.  The search of optimal weights in Hardy-type inequalities is receiving much attention because of its applications to functional analysis and its importance for the study of Dirichlet forms, spectral theory and probability, in particular properties of Brownian motion \cite{Fra17,KLW21,MP10}.

Regarding the relevance of Hardy inequalities in probability, it is well known that the {\textit{transiency}} of Brownian motion is closely related to the {\textit{subcriticality}} of the Laplace operator $-\Delta_{\mathbb{R}^d}:= -\sum_{i=1}^d \partial_i^2$ in $\mathbb{R}^d$. For instance, the Brownian motion is transient, i.e., the Brownian particle in $\mathbb{R}^d$ eventually escapes from any bounded set in $\mathbb{R}^d$ if and only if $d \geq 3$ (see \cite{MP10}), and in accordance with this, the Laplace operator $-\Delta_{\mathbb{R}^d}$ is subcritical in $\mathbb{R}^d$, i.e., there exists a function $W \gneq 0$ (will be referred to as a Hardy-weight) such that the Hardy-type inequality $-\Delta_{\mathbb{R}^d} \geq W$ holds (in the sense of quadratic forms) in $\mathbb{R}^d$ if and only if $d \geq 3$. In other words, for lower dimensions $d=1,2$, the Brownian motion is {\textit{recurrent}} (i.e., not transient), and $-\Delta_{\mathbb{R}^d}$ is {\textit{critical}} (i.e., not subcritical). This phenomenon is referred to as {\textit{criticality transition}} of the Laplacian in dimension \cite{GKS25}. On the other hand, the Laplace operator on any proper domain $\Omega$ of $\mathbb{R}^d$  (i.e., $\overline{\Omega}\neq \mathbb{R}^d$) is always subcritical \cite[Proposition 4.2]{PT_CVPDE}, \cite[Theorem 2.12]{K}. In particular, $-\Delta_{\mathbb{R}^d_+}$ is subcritical in the half space $\mathbb{R}^d_+:=\mathbb{R}^{d-1} \times (0,\infty)$ for all $d \geq 1$, i.e., no criticality transition occurs in this case. Moreover, in dimension one, it is known that any positive integer power of the Laplacian $(-\Delta_{\mathbb{R}_+})^k$ with $k \in \mathbb{N}$ is subcritical in $\mathbb{R}_+:=(0,\infty)$ {\cite{Birman}}. Also, the fractional powers, i.e., $(-\Delta_{\mathbb{R}_+})^{\sigma}$, $\sigma \in (0,1)$, are subcritical in $\mathbb{R}_+$, and in fact, explicit Hardy-weights were produced by \cite[$\sigma \neq 1/2$]{BD} and \cite[$\sigma = 1/2$]{ARS} (see also \cite[$\sigma = 1/2$]{AJR} for similar results on finite interval). We also refer to \cite{Herbst}, where similar Hardy-weights were obtained in case of the whole line for $(-\Delta_{\mathbb{R}})^{\sigma}$, $\sigma \in (0,1/2)$, which ensures that $(-\Delta_{\mathbb{R}})^{\sigma}$ is subcritical in $\mathbb{R}$ for $\sigma \in (0,1/2)$. However, unlike the half-line, $(-\Delta_{\mathbb{R}})^{\sigma}$ is critical in $\mathbb{R}$ for $\sigma \geq 1/2$ \cite{Herbst}.

 In the discrete case, the Laplace operator $-\Delta_{\mathbb{Z}^d}$ on $\mathbb{Z}^d$ follows the same criticality transition with respect to dimension as in the continuum case, i.e., $-\Delta_{\mathbb{Z}^d}$ is subcritical in $\mathbb{Z}^d$ if and only if $d \geq 3$. Recall that $-\Delta_{\mathbb{Z}^d}$ is defined as a difference operator on $\mathbb{Z}^d$, i.e., $$-\Delta_{\mathbb{Z}^d} f (n):= \sum_{|m-n|=1}(f(n)-f(m)) \,, \  \ \ f: \mathbb{Z}^d \rightarrow \mathbb{R} \,.$$    Also in dimension $d=1$, there is complete resemblance between the continuum and discrete settings in the context of criticality transition of fractional Laplacian. It is known that $(-\Delta_{\mathbb{Z}})^{\sigma}$ is subcritical if and only if $\sigma \in (0,1/2)$ \cite{GKS25,KN23}.   However, a contrasting phenomenon has recently been observed \cite{GKS25} on the discrete half-line $\mathbb{N}$. It is proven that $(-\Delta_{\mathbb{N}})^k$ with $k \in \mathbb{N}$ is subcritical on $\mathbb{N}$ if and only if $k=1$ \cite{GKS25}, while $(-\Delta_{\mathbb{R}_+})^k$ is subcritical for all $k \in \mathbb{N}$ as mentioned before. In fact, in \cite{GKS25} the authors considered any positive powers of the discrete Laplacian $(-\Delta_{\mathbb{N}})^{\sigma}$ on $\mathbb{N}$ with $\sigma >0$, and showed that the operator is subcritical if and only if $\sigma < \frac{3}{2}$. Their arguments heavily rely on the estimates of the Green kernel of the resolvent of $(-\Delta_{\mathbb{N}})^{\sigma}$ and the {\it Birman–Schwinger principle}. Although this approach does not immediately provide a Hardy-weight $W_{\sigma}$ for $(-\Delta_{\mathbb{N}})^{\sigma}$ with the {\it classical} decay as $W_{\sigma}(n) \asymp n^{-2\sigma}$ for large $n \in \mathbb{N}$, a more sophisticated approach with refined estimates of the Green kernel helps the authors to produce the following Hardy-weight 
 \begin{align} \label{Eq:HW}
     W_{\sigma}(n):= \begin{cases}
     \frac{\gamma}{n^{2\sigma}} \ \ 
 \ \ \ \ \ \ \  \mbox{if} \ \ \sigma \in (0,\frac{3}{2}) \setminus \{\frac{1}{2}\} \\
      \frac{\gamma}{n\log (1+n)} \ \ \mbox{if} \ \  \sigma=\frac{1}{2}
 \end{cases} \,,
 \end{align}
 where $\gamma$ is a positive constant which depends on $\sigma$.  Note that the Hardy-weight in \eqref{Eq:HW} has the desired decay except for $\sigma = \frac{1}{2}$.
In addition, the authors \cite{GKS25} do not quantify the optimal value of the constant $\gamma$ in \eqref{Eq:HW}, and in fact they left it as an open question. Moreover, they also propose the question of finding {\it critical Hardy-weight} $W$ for $(-\Delta_{\mathbb{N}})^{\sigma}$, $\sigma \in (0,\frac{3}{2})$.

 In this article, we address the questions raised above. Our main objective is to find {\it optimal Hardy-weights} $W_{\sigma}^{\mathrm{op}}$ for $(-\Delta_{\mathbb{N}})^{\sigma}$. Roughly speaking, an optimal Hardy-weight is a Hardy-weight for $(-\Delta_{\mathbb{N}})^{\sigma}$ on $\mathbb{N}$, which is 
`as large as possible'. Precisely, a Hardy-weight $W_{\sigma}^{\mathrm{op}}$ is called optimal if $(-\Delta_{\mathbb{N}})^{\sigma} - W_{\sigma}^{\mathrm{op}} \geq 0$ is both critical and null-critical in $\mathbb{N}$, see Section \ref{Sec:Criticality} for the definition of criticality and null-criticality.  
 
 The quest for optimal Hardy-weight for Schr\"odinger operators was originally proposed by Agmon \cite[Page 6]{Agmon}, who considered this problem in connection with his theory of
exponential decay of Schr\"odinger eigenfunctions, and answered first by Devyver--Frass--Pinchover \cite{DFP} in the continuum settings for second-order linear elliptic operators. Since then there is a significant interest in finding optimal Hardy-weight in various settings, for instance, see \cite{DP,Versano} for quasi-linear operators in the Euclidean space, \cite{BGG,DFP,FP,Fis_Rose} for Schr\"odinger operators on manifolds, and \cite{Fischer,KPP18} for the discrete graphs. For developments on more general settings, see \cite{CGL,Marcel,Tak}.

The key step in finding optimal Hardy-weights for an operator is an appropriate supersolution construction. To explain the main ideas briefly, let us consider the Laplace operator $-\Delta_{\mathbb{R}^d}$ on $\mathbb{R}^d_*:=\mathbb{R}^d \setminus \{0\}$. Any strictly positive supersolution $u$, i.e., $u>0,-\Delta_{\mathbb{R}^d}u \geq 0$ gives us a Hardy-weight $W_u=(-\Delta_{\mathbb{R}^d}u)/u$ for $-\Delta_{\mathbb{R}^d}$ due to Agmon--Allegretto--Piepenbrink theorem \cite[Theorem 4.3]{GY}. Now, a suitable choice of supersolution $u_0$ makes the corresponding operator $-\Delta_{\mathbb{R}^d} - W_{u_0}$ both critical and null-critical, and consequently, $W_{u_0}$ is an optimal Hardy-weight. For $d \geq 3$, one such choice of $u_0$ is the square root of the {\it positive minimal Green's function} $\mathcal{G}_{-\Delta_{\mathbb{R}^d}}^{1/2}$ of $-\Delta_{\mathbb{R}^d}$ in $\mathbb{R}^d_*$. 

Although the core idea remains the same in the discrete setting, in order to avoid certain technical difficulties in finding optimal Hardy-weights for {\it graph Laplacian} (Definition \ref{Def:graph-laplacian}), we require the underlying graph to be locally finite \cite{KPP18}. Indeed, since the standard graph on $\mathbb{Z}^d$, which is the graph on $X=\mathbb{Z}^d$ with \textit{standard weights} $b_{n,m}=\delta_{|n-m|,1}$ (see \cite[Definition 1.30]{KLW21}), is locally finite, an analogous construction as in the continuum case gives us an optimal Hardy-weights for $-\Delta_{\mathbb{Z}^d}$ when $d \geq 3$ \cite[Theorem 7.2]{KPP18}. These ideas also help the authors in \cite{KPPAMS} (see also \cite{DAFR}) to find an optimal Hardy-weight $W^{\mathrm{KPP}}$ in \eqref{og_Hardy}.

Now, concerning the fractional powers of the discrete Laplacian, it is well known that $(-\Delta_{\mathbb{Z}^d})^{\sigma}$ with $\sigma \in (0,1)$ gives rise to a graph Laplacian on $\mathbb{Z}^d$, see \cite{KN23}. In this view, to find an optimal Hardy-weight for $(-\Delta_{\mathbb{Z}^d})^{\sigma}$, one may anticipate to apply the general optimal Hardy-weight construction for a graph Laplacian given by \cite{KPP18}.  Unfortunately, this does not work as the underlying graph on $\mathbb{Z}^d$ corresponding to the fractional Laplacian $(-\Delta_{\mathbb{Z}^d})^{\sigma}$ is not locally finite, see \cite[(9)]{FRR24}. Nevertheless, for $d=1$, this problem has been tackled recently in \cite{KN23}. In this case, for the choice of a suitable supersolution, the authors took advantage of the following relation obtained in \cite{CR18}
    \begin{align} \label{Eqn:imp}
        (-\Delta_{\mathbb{Z}})^{\sigma} \mathcal{G}_{(-\Delta_{\mathbb{Z}})^{\alpha}}=\mathcal{G}_{(-\Delta_{\mathbb{Z}})^{\alpha-\sigma}} \ \ \ \mbox{on} \ \ \mathbb{Z}
    \end{align}
    for $\sigma<\alpha<\frac{1}{2}$, where $\mathcal{G}_{(-\Delta_{\mathbb{Z}})^{\sigma}}$ is the positive minimal Green's function of $(-\Delta_{\mathbb{Z}})^{\sigma}$. The above relation together with Agmon--Allegretto--Piepenbrink theorem ensures that
    $$\frac{(-\Delta_{\mathbb{Z}})^{\sigma} \mathcal{G}_{(-\Delta_{\mathbb{Z}})^{\alpha}}}{\mathcal{G}_{(-\Delta_{\mathbb{Z}})^{\alpha}}}$$
    are Hardy-weights for $(-\Delta_{\mathbb{Z}})^{\sigma}$ for all $\alpha \in (\sigma,1/2)$. Keller--Nietchsmann \cite{KN23} showed that these Hardy-weights are, in fact, critical if and only if $\alpha \leq \frac{1+2\sigma}{4}$. Moreover, this Hardy-weight is optimal if and only if $\alpha=\frac{1+2\sigma}{4}$. The optimal weight found in \cite{KN23} was first introduced in \cite{CR18}, and was constructed in a similar way, although it was not proven to be optimal.

    In order to construct optimal Hardy-weights for $(-\Delta_{\mathbb{N}})^{\sigma}$ with $\sigma \in (0,3/2)$, we adapt the ideas of \cite{KN23}, heavily relying on the spectral properties of $(-\Delta_{\mathbb{N}})^{\sigma}$ presented in \cite{GKS25}, and use standard tools from the discrete criticality theory developed in \cite{KPP17}. To do so, we first need to identify $(-\Delta_{\mathbb{N}})^{\sigma}$ as a graph Laplacian on $\mathbb{N}$ as the criticality theory in \cite{KPP17} is only valid for Schr\"odinger operators involving the graph Laplacian. Here comes the first obstacle. It turns out that $(-\Delta_{\mathbb{N}})^{\sigma}$ can be represented as a graph Laplacian {if and only if} $\sigma \leq 1$, see Section \ref{Sec:graph_Lap}. Henceforth, in our study of optimal Hardy-weights for $(-\Delta_{\mathbb{N}})^{\sigma}$, we need the restriction $\sigma \in (0,1]$. In this case, we have the following result.
    \begin{theorem}\label{Theorem:Hardy}
    Let $\sigma \in (0,1]$ and $\sigma<\alpha<1+\sigma$. Then
    $$W_{\alpha,\sigma}(n)=4^\sigma\frac{\Gamma(\frac{3}{2}-\alpha+\sigma)\Gamma(\alpha)}{\Gamma(\alpha-\sigma)\Gamma(\frac{3}{2}-\alpha)}\frac{\Gamma(n+\alpha-\sigma-1)\Gamma(n-\alpha+2)}{\Gamma(n-\alpha+\sigma+2)\Gamma(n+\alpha-1)}$$
    are critical Hardy-weights for $(-\Delta_\mathbb{N})^\sigma$ if and only if $\alpha \leq \frac{3+2\sigma}{4}$. In particular, for $\alpha=\frac{3+2\sigma}{4}$, 
    \begin{equation} \label{HW-Op_our}
        W_\sigma^{\mathrm{op}}(n):=W_{\frac{3+2\sigma}{4},\sigma}=4^{\sigma}\frac{\Gamma\left(\frac{3+2\sigma}{4}\right)^2}{\Gamma\left(\frac{3-2\sigma}{4}\right)^2}\frac{\Gamma(n-\frac{1+2\sigma}{4})\Gamma(n+\frac{5-2\sigma}{4})}{\Gamma(n+\frac{-1+2\sigma}{4})\Gamma(n+\frac{5+2\sigma}{4})}
    \end{equation}
 is an optimal Hardy-weight for $(-\Delta_\mathbb{N})^\sigma$.
\end{theorem}
As a byproduct of our result, we provide the best constant $\gamma$ in \eqref{Eq:HW}, see Section \ref{Sec:Final-Rem}\footnote{This sharp result has been found after the manuscript was published. In the published version we only proved an upper bound for the best constant}. We also note that the optimal Hardy-weight has the asymptotic $W_\sigma^{\mathrm{op}}(n) \asymp n^{-2\sigma}$ for large $n$ for every $\sigma \in (0,1]$, see Remark \ref{Rmk:weight-asympt}. This shows that there is no special behavior occurring at $\sigma =1/2$ in the context of optimal decay rate of the Hardy-weight for $(-\Delta_{\mathbb{N}})^{\sigma}$, $\sigma \in (0,1]$. Furthermore, it is worth mentioning that for $\sigma =1$ the optimal Hardy-weight given in \eqref{HW-Op_our} corresponds to
\begin{align*}
    W_1^{\mathrm{op}}(n) =\frac{1}{4}\frac{1}{n^2-\frac{9}{16}} =   \frac{1}{4n^2} + \frac{9}{64n^4} + \ldots \,,
\end{align*}
which improves the optimal Hardy-weight in \eqref{HW-Op_KPP} {\it asymptotically}, i.e., $W_1^{\mathrm{op}} > W^{\mathrm{KPP}}$ near infinity.

As we mentioned before, we follow the ideas from \cite{KN23} to prove Theorem \ref{Theorem:Hardy}. One of the crucial ingredients is to obtain an analogous relationship as in \eqref{Eqn:imp} for the operator $(-\Delta_{\mathbb{N}})^{\sigma}$, which is proved in Proposition \ref{Prop:Imp}. The proof of Proposition \ref{Prop:Imp} involves the refined estimates of the Green kernel of the resolvent of $(-\Delta_{\mathbb{N}})^{\sigma}$ obtained in \cite{GKS25}. For $\sigma \in (0,1)$, to show the criticality of $(-\Delta_\mathbb{N})^\sigma-W_{\alpha,\sigma}$, we use the standard techniques of constructing a null-sequence. In this approach, our arguments require $\sigma \in (0,1)$, and it fails for $\sigma =1$. For $\sigma =1$, we use the discrete {\it Liouville comparison principle} \cite{DKP24} to prove the criticality. The proof of null-criticality is rather straightforward due to the asymptotics of the Riesz potential $\mathcal{I}_{\alpha}$ (given by \eqref{Eq:Riesz}) obtained in Remark \ref{Eq:Asymp_Riesz}.

Next we discuss an immediate application of our study of optimal Hardy-weights in unique continuation theory.
Very recently, the study of optimal Hardy-weights together with the Liouville comparison principle has been used to establish certain unique continuation results for {\it positive} Schr\"odinger operators in the continuum \cite{DP24} and discrete \cite{DKP24} settings. More precisely, they gave a decay criterion that ensures when a solution to the Schr\"odinger equation associated to a positive Schr\"odinger operator with potentials bounded from above is trivial. In the literature, such results are known as the Landis-type unique continuation result as they concern the conjecture proposed by {Landis \cite{Landis}}. There is a huge bibliography on this topic. We refer to the recent comprehensive survey \cite{Survey} for this. Now, as we obtain optimal Hardy-weights in Theorem \ref{Theorem:Hardy}, we simply use the ideas from \cite{DKP24,DP24} to obtain the following unique continuation result. 
\begin{theorem} \label{Thm:UCP}
    Let $\sigma \in (0,1]$ and $V \leq 0$ except on a finite subset. Assume that $(-\Delta_{\mathbb{N}})^{\sigma} + V \geq 0$ and $\left[ (-\Delta_{\mathbb{N}})^{\sigma} + V\right]u  =0$ on $\mathbb{N}$ is such that 
    \begin{align*}
        |u| = O(n^{\sigma -\frac{1}{2}}) \,, \qquad \liminf_{n \rightarrow \infty} |u|n^{2-2\sigma}=0 \,.
    \end{align*}
    Then $u \equiv 0$.
\end{theorem}

The rest of this article is organised as follows. The next section focuses on the prerequisites that are required for the development of this article. In Section \ref{Optimal Hardy-weight}, we prove our main result Theorem \ref{Theorem:Hardy}. Theorem \ref{Thm:UCP} is proven in Section \ref{Sec:UCP}. Section \ref{Sec:Final-Rem} shows the sharpness of the constant and includes our concluding remarks. Some relevant computations are provided in Section \ref{comp}.


\section{Preliminaries}
In this section, we introduce the necessary concepts and definitions that are essential for this article. We collect some spectral properties of the fractional Laplacian on $\mathbb{N}$ proved in \cite{GKS25} and recall standard tools from criticality theory of graph Laplacian \cite{KPP17}. Throughout the article, we use the following notation and convention.
\begin{itemize}
    \item \textit{Kronecker delta function} $\delta_{n,m}=1$ if $n=m$ and $\delta_{n,m}=0$ otherwise.
    \item \textit{Chebyshev polynomials} of second kind $U_n(x)$ with $n\in \mathbb{N}$ and $x\in[-1,1]$.
    \item \textit{Gamma function} $\Gamma(x)=\int_0^\infty t^{x-1}e^{-t}dt$.
    \item \textit{Binomial coefficients} $\left(\begin{matrix} a \\ b \end{matrix} \right)=\frac{\Gamma(a+1)}{\Gamma(b+1)\Gamma(a-b+1)} $ with $a,b\in\mathbb{R}$.
    \item \textit{Pochhammer symbol} $(\beta)_k:=\beta(\beta+1)\ldots(\beta+k-1)$ with $\beta \in \mathbb{R}$.
    \item Since we can regard functions over $\mathbb{N}$ as sequences, we will use the notation $f(n)=f_n$ indistinctively to denote the value of $f$ at $n\in\mathbb{N}$.
    \item For positive functions $f$ and $g$ on $\mathbb{N}$, we denote $f\asymp g$ if there exists $C>0$ such that $C^{-1}f(n)\leq g(n)\leq Cf(n)$ for all $n \in \mathbb{N}$. 
    \item For positive functions $f$ and $g$ on $\mathbb{N}$, we denote $f\sim g$, if $\frac{f(n)}{g(n)} \rightarrow 1$ as $n \rightarrow \infty$.
    \item For functions $f$ and $g$ we denote $f\lesssim g $ if there exists a universal constant $C>0$ such that $f\leq Cg$.
    \item For a potential $ W: \mathbb{N} \to \mathbb{R} $, $\langle \cdot,\cdot \rangle_{\ell^2(\mathbb{N},W)}$ denotes the weighted $\ell^2$-inner product on $\mathbb{N}$ with respect to the weight $W$, i.e., $\langle \phi, \psi\rangle_{\ell^2(\mathbb{N},W)}= \sum_{\mathbb{N}} \phi(n) \psi(n) W(n)$. For $W\equiv 1$, we simply write $\ell^2(\mathbb{N})$ instead of $\ell^2(\mathbb{N},W)$.
    \item For $K \subseteq \mathbb{N} $, the characteristic function of $K$ is denoted by $1_K$. 
\end{itemize}

\subsection{The fractional Laplacian on half-line.} \label{Sec:graph_Lap}

We are going to take advantage of the fact that for a certain range of exponents $\sigma$ the fractional Laplacian $(-\Delta_{\mathbb{N}})^{\sigma}$ on $\mathbb{N}$ can be written as the Laplacian associated with a graph. This will enable us to use the theory of criticality of Schr\"odinger operators on graphs developed in \cite{KPP17}. {\it Graph Laplacian} operators are defined as follows (cf. \cite[Chapter 1.1]{KLW21}).

\begin{definition}[Graph Laplacian]\label{Def:graph-laplacian}
    Let $\{X,b,q\}$ be a graph with $X$ a countable set, $b: X\times X\rightarrow \left[0,\infty\right)$ satisfying
    \begin{itemize}
        \item $b_{n,m}=b_{m,n}$ for all $n,m\in X$.
        \item $b_{n,n}=0$ for all $n\in X$.
        \item $\sum_{n\in X} b_{n,m}<\infty$ for all $m\in X$.
    \end{itemize}
    and  $q: X\rightarrow \left[0,\infty\right)$.
    The {\it graph Laplacian} $\mathcal{L}_{X}$ associated with $\{X,b,q\}$ acting on 
    \begin{equation*}
        \mathcal{F}_{X}=\{f\in C(X)\,|\, \sum_{m\in X}b_{n,m}|f(n)|<\infty\,, \forall\  n\in X\}
    \end{equation*}
    is the operator
    \begin{equation}
        (\mathcal{L}_{X}f)(n):=\sum_{m\in X}b_{n,m}(f(n)-f(m))\,+\,q(n)f(n)\,.
    \end{equation}
\end{definition}
See \cite{KLW21} for general properties of the graph Laplacian and discrete Schr\"odinger operators. The Laplace operator on the discrete half-line $\mathbb{N}:=\{1,2,3,\ldots\}$ is defined as 
\begin{equation}\label{Discrete-Laplacian}
    (-\Delta_\mathbb{N})f(n):=-f(n-1)+2f(n)-f(n+1) \,,
\end{equation}
where $f:\mathbb{N}\rightarrow\mathbb{R}$ with the convention that $f(0)=0$, which can be thought of as a ``Dirichlet boundary condition''. Observe that this is the same operator as \cite{KPPAMS} on $\mathbb{N}$. Now we want to express the same operator as a graph Laplacian. 
Note that (\ref{Discrete-Laplacian}) together with the boundary condition can be written as a graph Laplacian over $X=\mathbb{N}$ 
\begin{equation} \label{gr_Lap_Rep}
     (-\Delta_{\mathbb{N}})f(n)=\sum_{m\in\mathbb{N}}b_{n,m}\big(f(n)-f(m)\big)\,+\, q_nf(n)
\end{equation}
with $b_{n,m}=1$ if $|n-m|=1$ and $b_{n,m}=0$ otherwise, and $q_n=\delta_{1,n}$. This graph is locally finite, i.e., for any fixed $n\in \mathbb{N}$, $b_{n,m}\neq0$ only for finitely many $m\in \mathbb{N}$. Observe that, when we view $(-\Delta_{\mathbb{N}})$ as a graph Laplacian in \eqref{gr_Lap_Rep}, we do not need to define $f(0)$. The boundary condition ($f(0)=0$) in \eqref{Discrete-Laplacian} is encoded in the term $q_n$ in \eqref{gr_Lap_Rep}.

Since $-\Delta_\mathbb{N}$ is a bounded self-adjoint operator on $\ell^2(\mathbb{N})$ with the spectrum $\sigma(-\Delta_\mathbb{N})=[0,4]$, we can define the positive fractional powers of $-\Delta_\mathbb{N}$ using its spectral representation, following the strategy described in \cite[Section 2.1]{GKS25}. Such a spectral representation is given in terms of Chebyshev polynomials of second kind $\{U_n\}_{\mathbb{N}\cup \{0\}}$, see for example \cite[Chapter 22]{AS} and \cite[Chapter 2]{Sze67} for the definition and properties of these polynomials. Observe that Chebyshev polynomials fulfill the recurrence relation
\begin{equation}
    \label{chev}
    U_{n+1}(x)-2xU_n(x)+U_{n-1}(x)=0
\end{equation}
with $U_0(x):=1$ and $U_1(x):=2x$. It is known that the set of normalised Chebyshev polynomials $\big\{\sqrt{\frac{2}{\pi}}U_n\ | \ n  \in \mathbb{N}\cup \{0\}\big\}$ constitutes an orthonormal basis of the Hilbert space $L^2((-1,1),\sqrt{1-x^2}\ \mathrm{d} x)$. Therefore we can define a map $\mathcal{U}$ such that
\begin{equation} \label{eq:def_U}
    \mathcal{U} e_n:=\sqrt{\frac{2}{\pi}}U_{n-1}
\end{equation}
with $e_n(m)=\delta_{n,m}$, and  extends to a unitary operator. We point out that in the context of \cite{CR18}, $\mathcal{U}$ will play the role of a Fourier transform.
Using (\ref{chev}) we get
\begin{equation}\label{spec}
    -\Delta_\mathbb{N}=\mathcal{U}^{-1}M_{2(1-x)}\mathcal{U}
\end{equation}
on $\ell^2(\mathbb{N})$ with $M_{2(1-x)}$ being the multiplication operator on $L^2((-1,1),\sqrt{1-x^2}\ \mathrm{d} x)$ by the function $x \mapsto 2(1-x)$. From (\ref{spec}), we get a spectral definition (cf. \cite[Definition A.21]{KLW21}) of the fractional Laplacian on $\mathbb{N}$ as follows
\begin{align} \label{eq:def_frac}
    (-\Delta_\mathbb{N})^\sigma=\mathcal{U}^{-1}M_{2^\sigma(1-x)^\sigma}\mathcal{U} \,,
\end{align}
for $\sigma >0$. 

In order to write $ (-\Delta_{\mathbb{N}})^\sigma$ as a graph Laplacian, we start from the explicit matrix elements obtained in \cite[Proposition 2.1]{GKS25}:
\begin{equation} \label{Eq:1}
    (-\Delta_{\mathbb{N}})^\sigma f(n) =\sum_{m \in \mathbb{N}} \mathcal{K}^\sigma_{m,n}f(m)  \,,
\end{equation}
where
\begin{align}\label{matrix-element}
\mathcal{K}^\sigma_{m,n}  &= \frac{2^{\sigma+1}}{\pi} \int_{-1}^1 (1 -x)^{\sigma}U_{m-1}U_{n-1} \sqrt{1-x^2} \ dx \nonumber \\
&=  
    (-1)^{m+n} \left[\left(\begin{matrix}
 2\sigma\\
\sigma+m-n
\end{matrix} 	
 \right)-\left(\begin{matrix}
 2\sigma\\
\sigma+m+n
\end{matrix} 	
 \right)\right].
\end{align}
Notice that $\mathcal{K}^\sigma_{m,n}=\mathcal{K}^\sigma_{n,m}$. Define $\widetilde{\mathcal{K}}^\sigma_{m,n}={\mathcal{K}}^\sigma_{m,n}$ if $m \neq n$ and $0$ otherwise.
Now we rewrite $ (-\Delta_{\mathbb{N}})^\sigma $ as follows
\begin{equation} \label{Eq:graph_Lap}
    (-\Delta_{\mathbb{N}})^\sigma f(n) =-\sum_{m \in \mathbb{N}} \widetilde{\mathcal{K}}^\sigma_{m,n}(f(n) -f(m) )+f(n) \sum_{m \in \mathbb{N}} \mathcal{K}^\sigma_{m,n} \,.
\end{equation}
If $-\widetilde{\mathcal{K}}^\sigma_{m,n}\geq 0$, this expression can be seen as the graph Laplacian of the graph $ \{\mathbb{N,}-\widetilde{\mathcal{K}}^\sigma,  R^\sigma\}$, with a potential term $$ R^\sigma_n :=\sum_{m \in \mathbb{N}} \mathcal{K}^\sigma_{m,n} \,.$$
Next proposition shows that this is possible when $\sigma \in (0,1]$. 

\begin{proposition} \label{Prop:graphLap}
   Let $\sigma \in (0,1]$. Then $\mathcal{K}^{\sigma}_{m,n} \leq 0$ for all $m\neq n$ \,. In particular,  $ \{\mathbb{N,}-\widetilde{\mathcal{K}}^\sigma,  R^\sigma\}$ is a graph Laplacian.
\end{proposition}
\begin{proof}
If $\sigma =1$, the conclusion follows easily from  \eqref{gr_Lap_Rep}. Now, consider $\sigma \in (0,1)$.   First notice that, by \eqref{matrix-element} and an application of \cite[Proposition 1]{LR18} (see Lemma \ref{LR18-comp} for details),
\begin{align*}
        \mathcal{K}^\sigma_{m,n}&=(-1)\left(\frac{(-1)^{m-n+1}\Gamma(2\sigma+1)}{\Gamma(\sigma+m-n+1)\Gamma(\sigma-m+n+1)}-\frac{(-1)^{m+n+1}\Gamma(2\sigma+1)}{\Gamma(\sigma+m+n+1)\Gamma(\sigma-m-n+1)}\right)\\
        &=c_\sigma\left(\frac{\Gamma(m+n-\sigma)}{\Gamma(m+n+\sigma+1)}-\frac{\Gamma(m-n-\sigma)}{\Gamma(m-n+\sigma+1)}\right)=c_\sigma(G(m+n)-G(m-n))
    \end{align*}
    with 
    \begin{equation*}
        c_\sigma= \frac{(-1)\Gamma(\sigma + 1/2)4^\sigma}{\sqrt{\pi}\Gamma(-\sigma)}
    \end{equation*}
    and $G(x)=\frac{\Gamma(x-\sigma)}{\Gamma(x+\sigma+1)}$. Note that for $\sigma \in (0,1)$, $c_\sigma>0$.
    
    Since $ \mathcal{K}^\sigma_{m,n}$ is a symmetric matrix, we assume, without loss of generality, that $m > n$. We now check that $G(x)$ is a decreasing function:
    \begin{align*}
        G'(x)&=\frac{\Gamma(x-\sigma)\psi^{(0)}(x-\sigma)\Gamma(x+\sigma+1)-\Gamma(x-\sigma)\psi^{(0)}(x+\sigma+1)\Gamma(x+\sigma+1)}{\Gamma(x+\sigma+1)^2}\\
        &=\frac{\Gamma(x-\sigma)\Gamma(x+\sigma+1)}{\Gamma(x+\sigma+1)^2}(\psi^{(0)}(x-\sigma)-\psi^{(0)}(x+\sigma+1))
    \end{align*}
    with $\psi^{(0)}(x)$ being the digamma function (see \cite[Chapter 6.3]{AS}). Since $\psi^{(0)}(x)$ is an increasing function for $x>0$ and $\sigma \in (0,1)$, it follows that $G'(x)<0$ for all $x \geq 1$. Hence, $\mathcal{K}^{\sigma}_{m,n} \leq 0$ for all $m\neq n$ \,.

 Now, we compute the potential term \begin{align}\label{Graph-potential}
     R^\sigma_n&:=\sum_{m \in \mathbb{N}} \mathcal{K}^\sigma_{m,n}=\sum_{m \in \mathbb{N}}(-1)^{m+n} \left[\left(\begin{matrix}
 2\sigma \nonumber\\
\sigma+m-n
\end{matrix} 	
 \right)-\left(\begin{matrix}
 2\sigma\\
\sigma+m+n
\end{matrix} 	
 \right)\right] \\ &=\sum_{m=-n+1}^\infty (-1)^{m} \left(\begin{matrix}
 2\sigma\\
\sigma+m
\end{matrix} 	
 \right) - \sum_{m=n+1}^\infty (-1)^{m} \left(\begin{matrix}
 2\sigma\\
\sigma+m
\end{matrix} 	
 \right)=\sum_{m=-n+1}^{n} (-1)^{m} \left(\begin{matrix}
 2\sigma\\
\sigma+m
\end{matrix} 	
 \right) \nonumber\\ &=-2 (-1)^n n \frac{\Gamma(2\sigma)}{\Gamma(1+\sigma-n)\Gamma(1+\sigma+n)}=\frac{2}{\pi}\frac{\Gamma(2\sigma)\Gamma(n-\sigma)}{\Gamma(1+\sigma+n)}n\sin(\pi\sigma),
\end{align}
where we have used Lemma \ref{potential-sum} in the second to last equality to compute the sum. From \eqref{Graph-potential} we can see that $R^{\sigma}$ is non-negative for $\sigma \in (0,1]$. This completes the proof of the proposition.
\end{proof}
\begin{rmk} \label{Rmk:graphLap} \rm
We make some important remarks here.
\begin{enumerate}[label=(\roman*)]
\item Note that \eqref{Eq:graph_Lap} is well defined if $f \in \mathcal{F}_{\mathbb{N}}^{\sigma}$, where
$$\mathcal{F}_{\mathbb{N}}^{\sigma}:=\{ f : \mathbb{N} \rightarrow \mathbb{R}: \sum_{m \in \mathbb{N}} \widetilde{\mathcal{K}}^\sigma_{m,n}|f(m)| < \infty \ \ \mbox{for each} \ n \in \mathbb{N}\} \,,$$
i.e., $\mathcal{F}_{\mathbb{N}}^{\sigma}$ is the domain of $(-\Delta_{\mathbb{N}})^{\sigma}$ when viewed as a graph Laplacian.
Using (\ref{Eq:Asmp_K}), one can verify that $\mathcal{F}_{\mathbb{N}}^{\sigma} \supseteq l^{\infty}(\mathbb{N})$. In particular, $\mathcal{F}_{\mathbb{N}}^{\sigma} \supseteq \ell^2(\mathbb{N})$.
    \item In fact $-\widetilde{\mathcal{K}}^{\sigma} $ is sign-changing if $\sigma > 1$. Indeed, in the proof above, if $1<\sigma<3/2$ then $c_\sigma<0$ and one can check that $G(m+n)-G(m-n)>0$ if $m-n=1$, and $G(m+n)-G(m-n)<0$ otherwise.  Therefore, the operator $(-\Delta_{\mathbb{N}})^{\sigma}$ can be viewed as a graph Laplacian if and only if $\sigma \in (0, 1]$. This is the main obstacle that prevents us from extending our result to $\sigma\in (1,3/2)$.
    \item Using the graph-Laplacian form of $(-\Delta)^{\sigma}$, i.e., \eqref{Eq:graph_Lap}, and the symmetry of the kernel, we can write down the associated quadratic form as follows.
\begin{align}\label{q_form}
\begin{split}
    \langle  f,(-\Delta_{\mathbb{N}})^\sigma f\rangle_{\ell^2(\mathbb{N})}&=-\sum_{n\in\mathbb{N}}\sum_{m\in\mathbb{N}}\mathcal{K}^\sigma_{m,n}(f_n -f_m )f_n +\sum_{n\in\mathbb{N}}R^\sigma_nf^2_n \\ &= -\frac{1}{2}\sum_{n\in\mathbb{N}}\sum_{m\in\mathbb{N}}\mathcal{K}^\sigma_{m,n}(f_m -f_n )^2+\sum_{n\in\mathbb{N}}R^\sigma_nf^2_n .
    \end{split}
\end{align}
    \item Observe that
   \begin{align*}
       \mathcal{K}^\sigma_{m,n}&=-c_\sigma\left(\frac{\Gamma(m-n-\sigma)}{\Gamma(m-n+\sigma+1)}-\frac{\Gamma(m+n-\sigma)}{\Gamma(m+n+\sigma+1)}\right)\\
       &=-c_\sigma\frac{\Gamma(m-n-\sigma)}{\Gamma(m-n+\sigma+1)}\left(1-\frac{\prod_{k=0}^{2n}(m-n-\sigma+k)}{\prod_{k=0}^{2n}(m-n+\sigma+1+k)}\right)
   \end{align*}
   In particular, since $\left(1-\frac{\prod_{k=0}^{2n}(m-n-\sigma+k)}{\prod_{k=0}^{2n}(m-n+\sigma+1+k)}\right)\leq1$, this means
   \begin{equation} \label{Eq:Asmp_K}
       |\mathcal{K}^\sigma_{m,n}|\lesssim |m-n|^{-2\sigma-1}.
   \end{equation}
\item In \eqref{Graph-potential}, using the well known asymptotic behaviour of the Gamma function
\begin{equation}
\label{eq:asympga}
    \frac{\Gamma(x+a)}{\Gamma(x+b)}\sim x^{a-b}
\end{equation}
for $a$ and $b$ real numbers \cite[5.11.12]{AR}, we see that the potential $R^{\sigma}$ decays as $n^{-2\sigma}$.
\item In view of \eqref{q_form} and \eqref{Eq:Asmp_K}, it is also relevant to study the following Hardy-type inequality
\begin{align} \label{H_var}
    \sum_{m,n \in \mathbb{N}, m\neq n} \frac{|f(m)-f(n)|^2}{|m-n|^{1+2\sigma}} \gtrsim \sum_{m \in \mathbb{N}} W(n) (f(n))^2 \,, \ \ \forall f \in C_c(\mathbb{N})
\end{align}
with suitable $W: \mathbb{N} \rightarrow [0,\infty) $.   We refer to the recent article \cite{Dy} for such a study. Note that the energy form $\sum_{m,n \in \mathbb{N}, m\neq n} \frac{|f(m)-f(n)|^2}{|m-n|^{1+2\sigma}}$ is different from $\langle (-\Delta_{\mathbb{N}})^{\sigma}  f,f\rangle_{\ell^2(\mathbb{N})}$  due to the sign changing $\mathcal{K}^{\sigma}$ and the presence of the potential $R^{\sigma}$ in the quadratic form $\langle (-\Delta_{\mathbb{N}})^{\sigma}  f,f\rangle_{\ell^2(\mathbb{N})}$. Moreover, it should be noted that the validity of the inequality \eqref{H_var} does not imply 
\begin{align} \label{H_og}
    \langle (-\Delta_{\mathbb{N}})^{\sigma}  f,f\rangle_{\ell^2(\mathbb{N})}\gtrsim \langle Wf,f\rangle_{\ell^2(\mathbb{N})} \,,
\end{align}
or vice versa. For instance, in \cite[Theorem 1.3]{Dy}, the author established \eqref{H_var} for all $\sigma >0$ with a specific $W$, but as we mentioned before, \eqref{H_og} does not hold for $\sigma \geq 3/2$ \cite[Theorem 1.1]{GKS25} with non-trivial non-negative $W$.
\end{enumerate}

\end{rmk}

For a {\em potential} $ W: \mathbb{N} \to \mathbb{R} $, we denote the corresponding Schr\"odinger operator 
$ \mathcal{H}_W $ acting on $ f\in \mathcal{F}_{\mathbb{N}}^{\sigma} $ as 
\begin{equation*}
	\mathcal{H}_W[f](n)=(-\Delta_{\mathbb{N}})^{\sigma} f(n)- {W(n)}f(n)\,,
\end{equation*}
and its associated energy functional as
\begin{equation}
\label{eq:quadratic}
	\mathcal{Q}_{W}(f,g):=\langle f,\mathcal{H}_W [g]\rangle_{\ell^2(\mathbb{N})}.
\end{equation}


Finally, we define a central object in our discussion, which are the sub/super solutions associated with Schr\"odinger operators.
\begin{definition}[({sub}/{super}) Solution]
A function $u \in \mathcal{F}_{\mathbb{N}}^{\sigma}$ is said to be a ({sub}/{super}) solution of the equation $\mathcal{H}_W[\varphi]=0$ in $\mathbb{N}$ if $\mathcal{H}_W[u]=0$ ($\mathcal{H}_W[u] {\leq}  0$ / $\mathcal{H}_W[u] {\geq}  0$) in $\mathbb{N}$.
\end{definition}

Agmon--Allegretto--Piepenbrink-type theorem relates the existence of positive supersolution to $\mathcal{H}_W[\varphi]=0$ in $\mathbb{N}$ with the positivity of the functional 
$$\mathcal{Q}_{W}(\varphi) := \langle  (-\Delta_{\mathbb{N}})^{\sigma}\varphi, \varphi  \rangle_{\ell^{2}(\mathbb{N})} - \langle W \varphi,\varphi \rangle_{{\ell^{2}(\mathbb{N})}} $$
on $C_{c}(\mathbb{N})$. For a reference on a general graph, see \cite[Theorem 4.2]{KPP17} and \cite[Theorem 2.3]{Florian}. We state it below for our particular case.
\begin{proposition}[Agmon--Allegretto--Piepenbrink-type theorem] \label{Prop:AAP}
Let $\sigma \in (0,1]$ and $W:\mathbb{N} \rightarrow \mathbb{R}$ be a potential. Then $\mathcal{Q}_{W}(\varphi)  \geq 0$ on $C_{c}(\mathbb{N})$ if and only if the equation $$\mathcal{H}_W[\varphi]:=((-\Delta_{\mathbb{N}})^{\sigma} -W)[\varphi]=0$$ admits a strictly positive supersolution on $\mathbb{N}$.
\end{proposition}

\subsection{Criticality theory for graph Laplacian.} \label{Sec:Criticality}

We have seen in the previous section that the fractional Laplacian $(-\Delta_{\mathbb{N}})^{\sigma}$ is a graph Laplacian associated with the graph $-\widetilde{\mathcal{K}}^{\sigma}$ over $\mathbb{N}$ with the non-negative potential term $R^{\sigma}$ as long as $-\widetilde{\mathcal{K}}^{\sigma} \geq 0$. It turns out that this is the case if $\sigma \in (0,1]$ (Proposition \ref{Prop:graphLap}). This allows us to take advantage of the well-developed criticality theory for graph Laplacian, see \cite{KPP17}. In this section, we recall some of the notions and results from criticality theory that are required for our purpose.


In view of Proposition \ref{Prop:AAP}, we define the property of possessing a strictly positive supersolution as the positivity of the underlying operator. We also give the definition of criticality and subcriticality.
\begin{definition}[Positivity of $\mathcal{H}_W$, Sub-criticality, Criticality] Let $\sigma \in (0,1]$ and $W:\mathbb{N} \rightarrow \mathbb{R}$ be a potential.

\begin{enumerate}[label=(\roman*)]
    \item We say that the Schr\"odinger operator $\mathcal{H}_W:=(-\Delta_{\mathbb{N}})^{\sigma} -W$ is positive and write $\mathcal{H}_W \geq 0$ in $\mathbb{N}$ if $\mathcal{H}_W[\varphi]=0$ admits a strictly positive supersolution on $\mathbb{N}$.
    \item $\mathcal{H}_W \geq 0$ is said to be {\it subcritical} if $\mathcal{H}_W - \widetilde{W}  \geq 0$ for some $ \widetilde{W} \gneq 0$ and in this case $\widetilde{W}$ is referred to as a Hardy-weight of $\mathcal{H}_W$.
    \item $\mathcal{H}_W \geq 0$ is called {\it critical} if it is not subcritical.
\end{enumerate}
\end{definition}

Next we recall a characterization of criticality which will be used to prove Theorem \ref{Theorem:Hardy}. See \cite[Theorem 5.3]{KPP17} for a reference.
\begin{proposition}[Criticality $\&$ Null-sequence] \label{Prop:char_cri}
Let $\sigma \in (0,1]$ and $W:\mathbb{N} \rightarrow \mathbb{R}$ be a potential. $\mathcal{H}_W \geq 0$ is critical if and only if it admits a {\it null-sequence}, i.e., $\exists (\phi_k)$ in $C_c(\mathbb{N})$ such that $\phi_k \geq 0$,  $\phi_k(n) \asymp 1$ for some fixed $n \in \mathbb{N}$, and $$\mathcal{Q}_W(\phi_k) \rightarrow 0\quad \text{ as } k \rightarrow \infty.
$$    \end{proposition}
\begin{rmk} \rm 
  It turns out that any null-sequence of $\mathcal{H}_W$ converges pointwise, and the pointwise limit of any two null-sequences are positive multiple of each other. We call the pointwise limit of a null-sequence of $\mathcal{H}_W$ an {\it Agmon ground state} of $\mathcal{H}_W$.  Up to scalar multiples, Agmon ground state is a unique strictly positive
 super-solution of $\mathcal{H}_W[\varphi]=0$ on $\mathbb{N}$,
 see \cite[Theorem 5.3]{KPP18}. Thus, we see that $\mathcal{H}_W$ is critical in $\mathbb{N}$ if and only if it admits an Agmon ground state.   We also remark that an Agmon ground state $u$ has {\it minimal growth at infinity} in $\mathbb{N}$, i.e., for any $v>0$ with $\mathcal{H}_W[v] \geq 0$ on $\mathbb{N} \setminus K$ for some finite subset $K$, we have $u \leq v$ on $K$ implies $u \leq v$ on $\mathbb{N}\setminus K$ \cite[Theorem 2.6]{Florian}.

On the other hand, $\mathcal{H}_W$ is sub-critical in $\mathbb{N}$ if and only if it admits a positive {\it minimal Green function}, i.e., $\mathcal{H}_W[\varphi]=\delta_{1,\cdot}$ on $\mathbb{N}$ admits a strictly positive solution which has minimal growth at infinity \cite[Theorem 2.5]{Florian}.
\end{rmk}

Now we define the optimality of a Hardy-weight as follows.
\begin{definition}[Optimal Hardy-weight]
 Let $\sigma \in (0,1]$.   A Hardy-weight $W$ of $(-\Delta)^{\sigma}$ is said to be an \it{optimal Hardy-weight} if the following conditions hold:
    \begin{enumerate}[label=(\roman*)]
        \item $\mathcal{H}_W:=(-\Delta)^{\sigma}-W$ is critical;
        \item $\Phi \notin \ell^2(\mathbb{N},W)$, where $\Phi$ is an Agmon ground state of $\mathcal{H}_W$.
    \end{enumerate}
\end{definition}
\begin{rmk} \rm\label{Rmk:Asympt-Optimal}
The second criterion in the above definition is referred to as {\it null-criticality of $\mathcal{H}_W$}. This condition ensures that the Hardy-weight $W$ is in fact `big' near infinity. In particular it implies that the weight is {\it optimal near infinity}, i.e, if for some $\lambda\geq0$
\begin{equation*}
    \langle (-\Delta)^\sigma f,f\rangle_{\ell^2(\mathbb{N)}}\geq (1+\lambda)\langle W f,f\rangle_{\ell^2(\mathbb{N)}}
\end{equation*}
for all $f\in C_c(\mathbb{N}\setminus K)$ for some compact $K \subseteq \mathbb{N}$, then $\lambda =0$. It is worth mentioning that this condition will ensure that the constant implicitly involved in the Hardy-weight $W_{\sigma}^{\operatorname{op}}$ is optimal.
    
\end{rmk}


\subsection{Liouville-type comparison principle.} \label{Sec:LC} The following Liouville-type comparison principle was proved in \cite[Theorem 2.1]{DKP24}, and we include here for the sake of completeness. The continuum analogue was proved earlier by Pinchover \cite{Pinchover_LC}. This principle provides us a way to conclude the criticality of a positive Schrödinger operator from that of another positive Schrödinger operator when they satisfy certain compatible conditions for comparison.
\begin{theorem}[Liouville-type comparison principle]\label{Thrm_Liouville}
    Suppose $\mathcal{H}_W$ and $\mathcal{H}_{W'}$ are positive
Schrödinger operators with potentials $W$, $W'$ over $\mathbb{N}$. Let $u,v \in\mathcal{F}_{\mathbb{N}}^{\sigma}$  be such that
\begin{enumerate}[label=(\alph*)]
\item  $\mathcal{H}_{W'}$ is critical and $v > 0$ is its Agmon ground state,
\item   $\mathcal{H}_{W}[u_+] \leq 0$ with $u_+\neq 0$,
\item $\exists C>0$ such that  $u_+(n)  \leq C v(n)$ for all $n \in \mathbb{N}$.
\end{enumerate}
Then, $\mathcal{H}_W$ is critical and $u > 0$ is an Agmon ground state of $\mathcal{H}_W$.
\end{theorem}

\section[Optimal Hardy-weight]{Optimal Hardy-weight for $(-\Delta_{\mathbb{N}})^\sigma$: Proof of Theorem \ref{Theorem:Hardy}}  \label{Optimal Hardy-weight}   
In this section, we give a proof of Theorem \ref{Theorem:Hardy}. The proof has three major steps:
\begin{enumerate}
    \item Construct a positive supersolution of $(-\Delta_{\mathbb{N}})^\sigma[\varphi]=0$ in $\mathbb{N}$ for $\sigma \in (0,1]$ and use it to apply a ground state transformation of the quadratic form associated with $(-\Delta_{\mathbb{N}})^\sigma$, which readily gives a family of Hardy-weights.
    \item From the family of Hardy-weights obtained in Step-1, find the ones that are critical.
    \item Among the critical Hardy-weights obtained in Step-2, find those which are also null-critical.
\end{enumerate}

\subsection{Supersolution construction.}
First we look for a supersolution of the equation $(-\Delta_{\mathbb{N}})^\sigma[\varphi]=0$ in $\mathbb{N}$. For this purpose, we consider the Riesz potential defined as
\begin{align} \label{Eq:Riesz}
    \mathcal{I}_\alpha (n) :=\langle \mathcal{U} e_n,[2^{-\alpha}[1-\cdot]^{-\alpha}]\rangle_{L^2((-1,1),\sqrt{1-x^2}\mathrm{d} x)} \,, \ \  \  \ \alpha \in {(0,\frac{3}{2})} \,.
\end{align}
Note that $\mathcal{I}_\alpha$ is finite for each $n$ as $\alpha <\frac{3}{2}$ and it is a strictly positive function. Observe that 
\begin{align}\label{riesz}
\begin{split}
    \mathcal{I}_\alpha (n)&=\langle \mathcal{U} e_n,[2^{-\alpha}[1-\cdot]^{-\alpha}]\rangle_{L^2((-1,1),\sqrt{1-x^2}\mathrm{d} x)}\\ 
    &=\sqrt{\frac{2}{\pi}}\int_{-1}^1 \frac{U_{n-1}(x)}{2^{\alpha}(1-x)^{\alpha}}\sqrt{1-x^2}\ \mathrm{d}x  \\ &= \sqrt{\frac{2}{\pi}} {\frac{\pi}{2^{\alpha}2^{-\alpha+1}}}(-1)^{n+1}\Gamma(1-2\alpha)\left[\frac{1}{\Gamma(n-\alpha)\Gamma(2-\alpha-n)}-\frac{1}{\Gamma(-n-\alpha)\Gamma(2-\alpha+n)}\right] \\ &=-\sqrt{{\frac{\pi}{2}}}\left[\frac{\Gamma(1/2-\alpha)4^{-\alpha}\Gamma(\alpha-n-1)}{\Gamma(-n-\alpha)\sqrt{\pi}\Gamma(\alpha)}-\frac{\Gamma(1/2-\alpha)4^{-\alpha}\Gamma(\alpha+n-1)}{\Gamma(n-\alpha)\sqrt{\pi}\Gamma(\alpha)}\right]\\&=-\frac{1}{4^\alpha{\sqrt{2}}}\frac{\Gamma(1/2-\alpha)}{\Gamma(\alpha)}\left[\frac{\Gamma(\alpha-n-1)}{\Gamma(-n-\alpha)}-\frac{\Gamma(\alpha+n-1)}{\Gamma(n-\alpha)}\right]\\
    &=-\frac{1}{4^\alpha{\sqrt{2}}}\frac{\Gamma(1/2-\alpha)}{\Gamma(\alpha)}\frac{\Gamma(n+\alpha-1)}{\Gamma(n-\alpha+2)}2n(2\alpha-1)=\frac{1}{4^\alpha{\sqrt{2}}}\frac{\Gamma(3/2-\alpha)}{\Gamma(\alpha)}\frac{\Gamma(n+\alpha-1)}{\Gamma(n-\alpha+2)}4n >0 
    \end{split}
\end{align}
for all $\alpha \in (0,\frac{3}{2})$, where we used  \cite[Lemma A.1]{GKS25}, \cite[(3.1)]{LR18} and Lemma \ref{simpriesz} in the second, third and second to last equality, respectively.

\begin{rmk} \label{Eq:Asymp_Riesz} \rm
By \eqref{eq:asympga} we get the following 
asymptotic behavior of $\mathcal{I}_\alpha$ as $n \rightarrow \infty$:
\begin{align} \label{Eq:Asmp_I}
{  \mathcal{I}_\alpha (n)  {\sim} \frac{2{\sqrt{2}}}{4^\alpha}\frac{\Gamma(3/2-\alpha)}{\Gamma(\alpha)}n^{2\alpha-2} \,. }
\end{align}
For $\alpha \in (0,1]$, $\mathcal{I}_{\alpha}$ is in fact a positive  minimal Green function (up to a multiplicative constant) of $(-\Delta_{\mathbb{N}})^{\alpha}$ as it is shown in Section \ref{Sec:UCP}.
\end{rmk}

The following proposition plays a fundamental role in the supersolution construction which is the key for proving our main result. Note that this is the version of \cite[Proposition 2]{KN23} in the half-line $\mathbb{N}$. In \cite[Proposition 2]{KN23}, the authors approximate their corresponding Riesz potential by certain square summable functions. These approximating functions are positive which enables the authors to use monotone convergence theorem to justify a crucial limiting argument in their proof. In our case, such approximating functions of $\mathcal{I}_{\alpha}$ change their sign which causes a difficulty in the limiting argument. We overcome this by analyzing the asymptotic behavior of $\mathcal{I}_{\alpha}$.    
\begin{proposition} \label{Prop:Imp}
 Let $\sigma, \alpha \in (0,3/2)$ and $\sigma < \alpha < 1+\sigma$. Then
 $$(-\Delta_{\mathbb{N}})^{\sigma} \mathcal{I}_{\alpha}=\mathcal{I}_{\alpha-\sigma} \,,$$
 where $\mathcal{I}_{\alpha}$ is given by \eqref{Eq:Riesz}.
 In particular, $\mathcal{I}_{\alpha}$ is a strictly positive supersolution of $(-\Delta_{\mathbb{N}})^{\sigma} \mathcal{I}_{\alpha}=0 $ in $\mathbb{N}$.
\end{proposition}
\begin{proof}
For $\varepsilon>0$, we consider
$$\mathcal{I}_\alpha^{\varepsilon} (n) =\langle \mathcal{U} e_n,[2^{-\alpha}[1+\varepsilon-\cdot]^{-\alpha}]\rangle_{L^2((-1,1),\sqrt{1-x^2}\mathrm{d} x)}  \,.$$
Since the function $[1+\varepsilon-\cdot]^{-\alpha} \in L^2((-1,1),\sqrt{1-x^2}\mathrm{d} x)$, we infer that $\mathcal{U} \mathcal{I}_\alpha^{\varepsilon} = 2^{-\alpha}[1+\varepsilon-\cdot]^{-\alpha}.$ Moreover, if $\alpha < \frac{3}{2}$, an application of dominated convergence theorem implies that $\mathcal{I}_\alpha^{\varepsilon} \rightarrow \mathcal{I}_\alpha$ pointwise as $\varepsilon \rightarrow 0$. Now, we see that
  \begin{align}
   (-\Delta_{\mathbb{N}})^{\sigma} \mathcal{I}_{\alpha}^{\varepsilon}(n) & = \langle  e_n, (-\Delta_{\mathbb{N}})^{\sigma} \mathcal{I}_{\alpha}^{\varepsilon} \rangle_{\ell^2(\mathbb{N})} =    \langle  e_n, [\mathcal{U}^{-1} M_{2^{\sigma}(1-\cdot)^{\sigma}} \mathcal{U}] \mathcal{I}_{\alpha}^{\varepsilon} \rangle_{\ell^2(\mathbb{N})} \nonumber \\
   & = \langle   \mathcal{U} e_n, M_{2^{\sigma}(1-\cdot)^{\sigma}} \mathcal{U} \mathcal{I}_{\alpha}^{\varepsilon} \rangle_{L^2((-1,1),\sqrt{1-x^2})} \nonumber \\
   & = \langle   \mathcal{U} e_n, \frac{2^{\sigma}(1-\cdot)^{\sigma}}{ 2^{\alpha}(1+\varepsilon-\cdot)^{\alpha}} \rangle_{L^2((-1,1),\sqrt{1-x^2})} \nonumber \\
   &  \rightarrow \mathcal{I}_{\alpha-\sigma}(n) \label{lim1}
  \end{align}
  as $\varepsilon \rightarrow 0$, if $\alpha-\sigma<\frac{3}{2}$. On the other hand, 
  \begin{align} \label{Eqn:DCT0}
      (-\Delta_{\mathbb{N}})^{\sigma} \mathcal{I}_{\alpha}^{\varepsilon}(n) = \sum_{m \in \mathbb{N}} {\mathcal{K}^{\sigma}_{m,n} }\mathcal{I}_{\alpha}^{\varepsilon}(m) \,.
  \end{align}
  Recall that $\mathcal{I}_{\alpha}^{\varepsilon} \rightarrow \mathcal{I}_{\alpha}$ pointwise as $\varepsilon \rightarrow 0$. Furthermore, choosing $\beta \in (2\alpha-3/2,1/2+2\sigma)\cap (0,3/2)$  (which is non-empty as $\alpha<1+\sigma$), we use the Cauchy-Schwarz inequality to obtain
\begin{align*}
|\mathcal{I}_{\alpha}^{\varepsilon} (m)| & \leq  2^{-\alpha} \sqrt{\frac{2}{\pi}} \int_{-1}^1 \frac{|U_{m-1}|}{(1+\varepsilon-x)^{\alpha}} \sqrt{1-x^2} dx \\
& \leq 2^{-\alpha} \sqrt{\frac{2}{\pi}} \left[\int_{-1}^1 \frac{U_{m-1}^2}{(1+\varepsilon-x)^{\beta}} \sqrt{1-x^2} dx \right]^{1/2} \left[\int_{-1}^1 \frac{\sqrt{1-x^2}}{(1+\varepsilon-x)^{2\alpha-\beta}}  dx \right]^{1/2} \\
& \leq 2^{-\alpha} \sqrt{\frac{2}{\pi}} \left[\int_{-1}^1 \frac{U_{m-1}^2}{(1-x)^{\beta}} \sqrt{1-x^2} dx \right]^{1/2} \left[\int_{-1}^1 \frac{\sqrt{1-x^2}}{(1-x)^{2\alpha-\beta}}  dx \right]^{1/2} \\
& := C_{\alpha,\beta} \sqrt{J_{\beta}(m)} \,,
\end{align*} 
where $C_{\alpha,\beta}$ is a finite positive constant (as the second integral in the last inequality is finite for $\beta>2\alpha-3/2$) depending only on $\alpha,\beta$ and does not depend on $m$, and 
$$J_{\beta}(m):=\int_{-1}^1 \frac{U_{m-1}^2}{(1-x)^{\beta}} \sqrt{1-x^2} dx \,.$$
It has been shown in \cite[Lemma A.2]{GKS25} that 
$$J_{\beta}(m)=2^{\beta-2} \frac{\Gamma(\beta)^2}{\Gamma(2\beta)} j_{\beta}(m) \,,$$
where \begin{align} \label{Eq:asymp}
    j_{\beta}(m):=\left( 1-\frac{(\beta)_{2m}}{(1-\beta)_{2m}} \tan(\pi \beta)\right) = \begin{cases}
    O(1) \ \ \mbox{if} \ \ \beta \in (0,\frac{1}{2}) \,, \\
    O(\log m) \ \ \mbox{if} \ \ \beta =\frac{1}{2} \,, \\
    O(m^{2\beta-1}) \ \ \mbox{if} \ \ \beta \in (\frac{1}{2},\frac{3}{2}) \,, 
\end{cases} 
\end{align}
and $(\cdot)_{2m}$ is the Pochhammer symbol.
This estimate together with the estimate \eqref{Eq:Asmp_K} of $\mathcal{K}_{\cdot, n}^{\sigma}$, we conclude that for each $n \in \mathbb{N}$, the function $\mathcal{K}^{\sigma}_{\cdot,n} \mathcal{I}_{\alpha}^{\varepsilon}(\cdot)$ is dominated outside a finite set by the integrable function
$$ \frac{\sqrt{J_{\beta}(\cdot)}}{|\cdot-n|^{1+2\sigma}} \,.$$
Thus,  using the dominated convergence theorem in \eqref{Eqn:DCT0}, we get
   \begin{align} \label{lim2}
      (-\Delta_{\mathbb{N}})^{\sigma} \mathcal{I}_{\alpha}^{\varepsilon}(n) = \sum_{m \in \mathbb{N}} {\mathcal{K}^{\sigma}_{m,n} }\mathcal{I}_{\alpha}^{\varepsilon}(m) {\rightarrow} \sum_{m \in \mathbb{N}} {\mathcal{K}^{\sigma}_{m,n} }\mathcal{I}_{\alpha}(m) = (-\Delta_{\mathbb{N}})^{\sigma} \mathcal{I}_{\alpha}(n) \,,
  \end{align}
  as $\varepsilon \rightarrow 0$. Thus, it follows from \eqref{lim1} and \eqref{lim2} that $(-\Delta_{\mathbb{N}})^{\sigma} \mathcal{I}_{\alpha}=\mathcal{I}_{\alpha-\sigma} \,.$ This proves the first part of the proposition. The second conclusion is immediate from the fact that both $\mathcal{I}_{\alpha},\mathcal{I}_{\alpha-\sigma}$ are strictly positive as shown in \eqref{riesz}.
\end{proof}

We consider the following function, for $0<\alpha, \sigma<3/2$,
$$W_{\alpha,\sigma}:=\frac{\mathcal{I}_{\alpha-\sigma}}{\mathcal{I}_{\alpha}} \,, \ \ \ 0<\alpha-\sigma<\frac{3}{2} \,.$$
Clearly, from \eqref{riesz}, it can be seen that $W_{\alpha,\sigma}>0$ for any $0<\alpha, \sigma<3/2$ with $0<\alpha-\sigma<\frac{3}{2}$. The forthcoming proposition shows that $W_{\alpha,\sigma}$ is in fact a Hardy-weight for $(-\Delta_{\mathbb{N}})^{\sigma}$ when $\sigma \in (0,1]$, for certain values of $\alpha$.
\begin{proposition}[Construction of Hardy-weights] \label{Prop:HW}
    Let $\sigma \in (0,1]$, $\alpha \in (0,3/2)$ and $\sigma<\alpha<1+\sigma$. Then $W_{\alpha,\sigma}$ is a  Hardy-weight for $(-\Delta_{\mathbb{N}})^{\sigma}$.
\end{proposition}
\begin{proof}
 It is an immediate consequence of Proposition \ref{Prop:AAP}. Indeed, one can see that $\mathcal{I}_{\alpha}$ is a positive solution of 
 $$\left((-\Delta_{\mathbb{N}})^{\sigma}-W_{\alpha,\sigma}\right)[\varphi]=0 \,,$$
 which by Proposition \ref{Prop:AAP} implies $(-\Delta_{\mathbb{N}})^{\sigma}-W_{\alpha,\sigma} \geq 0$.
\end{proof}
Observe that the above proposition ensures the positivity of $(-\Delta_{\mathbb{N}})^{\sigma}-W_{\alpha,\sigma}$ qualitatively. In the next proposition, we quantify this positivity by proving a {\it Hardy-type identity}, which can be viewed as a {\it ground state representation} of $(-\Delta_{\mathbb{N}})^{\sigma}$ [cf. \cite{FLS06}]. We mainly follow the ideas from \cite[Theorem 2.1]{CR18}.
\begin{proposition}[Ground state representation] \label{Prop:HW_Alt}  Let $\sigma, \alpha \in (0,3/2)$ and $\sigma<\alpha<1+\sigma$. Then
    \begin{equation}
    \langle  [(-\Delta_{\mathbb{N}})^\sigma-W_{\alpha,\sigma}] \varphi, \varphi \rangle_{\ell^2(\mathbb{N})} = \sum_{n\in\mathbb{N}}\sum_{m\in\mathbb{N}}-\mathcal{K}^\sigma_{m,n}\left(\frac{\varphi(n) }{\mathcal{I}_{\alpha}(n) }-\frac{\varphi(m) }{\mathcal{I}_{\alpha}(m) }\right)^2 \mathcal{I}_{\alpha}(n) \mathcal{I}_{\alpha}(m)  \,, \ \ \forall \varphi \in C_c(\mathbb{N}) \,,
\end{equation}
where $\mathcal{K}^\sigma_{m,n}$ is given by \eqref{matrix-element}.
In particular, if $\sigma \in (0,1]$ then $W_{\alpha,\sigma}$ is a Hardy-weight of $(-\Delta_{\mathbb{N}})^\sigma$.
\end{proposition}
\begin{proof}
For a given $\varphi \in C_c(\mathbb{N})$, we take $u(n) =\varphi(n)^2 /\mathcal{I}_{\alpha}(n) $. Then
\begin{equation}\label{weight-construction}
     \langle (-\Delta_{\mathbb{N}})^\sigma u,\mathcal{I}_{\alpha}\rangle_{\ell^2(\mathbb{N})}= \langle u,\mathcal{I}_{\alpha-\sigma}\rangle_{\ell^2(\mathbb{N})}=\sum_{n\in\mathbb{N}}\frac{\mathcal{I}_{\alpha-\sigma}}{\mathcal{I}_\alpha} \varphi(n)^2.
\end{equation}
We can polarize \eqref{q_form} with $f=u$ and $g=\mathcal{I}_{\alpha}$  to get
\begin{align*}\label{ground}
\begin{split}
& \, \langle  f,(-\Delta_{\mathbb{N}})^\sigma g\rangle_{\ell^2(\mathbb{N})}=-\frac{1}{2}\sum_{n\in\mathbb{N}}\sum_{m\in\mathbb{N}}\mathcal{K}^\sigma_{m,n}(f_m -f_n )(g_m -g_n )+\sum_{n\in\mathbb{N}}R^\sigma_nf_n g_n \\&=-\frac{1}{2}\sum_{n\in\mathbb{N}}\sum_{m\in\mathbb{N}}\mathcal{K}^\sigma_{m,n}\left((\varphi(n) -\varphi(m) )^2-\left(\frac{\varphi(n) }{\mathcal{I}_{\alpha}(n) }-\frac{\varphi(m) }{\mathcal{I}_{\alpha}(m) }\right)^2 \mathcal{I}_{\alpha}(n) \mathcal{I}_{\alpha}(m) \right)+\sum_{n\in\mathbb{N}}R^\sigma_n \varphi(n)^2 \\&= \langle  \varphi,(-\Delta_{\mathbb{N}})^\sigma \varphi \rangle_{\ell^2(\mathbb{N})}+\frac{1}{2}\sum_{n\in\mathbb{N}}\sum_{m\in\mathbb{N}}\mathcal{K}^\sigma_{m,n}\left(\frac{\varphi(n) }{\mathcal{I}_{\alpha}(n) }-\frac{\varphi(m) }{\mathcal{I}_{\alpha}(m) }\right)^2 \mathcal{I}_{\alpha}(n) \mathcal{I}_{\alpha}(m) \\& =\sum_{n\in\mathbb{N}}\frac{\mathcal{I}_{\alpha-\sigma}}{\mathcal{I}_\alpha} \varphi(n)^2\,,
\end{split}
\end{align*}
where we used \eqref{weight-construction} in the last equality. Consequently, we obtain the following ground state representation of $(-\Delta)^{\sigma}$
\begin{equation} \label{Eq:GSR}
    \langle  \varphi,[(-\Delta_{\mathbb{N}})^\sigma-W_{\alpha,\sigma}] \varphi\rangle_{\ell^2(\mathbb{N})}=\frac{1}{2} \sum_{n\in\mathbb{N}}\sum_{m\in\mathbb{N}}-\mathcal{K}^\sigma_{m,n}\left(\frac{\varphi(n) }{\mathcal{I}_{\alpha}(n) }-\frac{\varphi(m) }{\mathcal{I}_{\alpha}(m) }\right)^2 \mathcal{I}_{\alpha}(m) \mathcal{I}_{\alpha}(n) 
\end{equation}
with 
\begin{equation}
    \label{weight}
    W_{\alpha,\sigma}(n)=\frac{\mathcal{I}_{\alpha-\sigma}}{\mathcal{I}_\alpha}.  
\end{equation}

Since $\mathcal{K}^\sigma_{m,n} \leq 0$ for $m \neq n$ and $\sigma \in (0,1]$ (Proposition \ref{Prop:graphLap}), it follows from \eqref{Eq:GSR} that $W_{\alpha,\sigma}$ is a Hardy-weight of $(-\Delta_{\mathbb{N}})^{\sigma}$.
\end{proof}

\begin{rmk} \rm\label{Rmk:weight-asympt}
We can infer the decay of the Hardy-weight $W_{\alpha,\sigma}$ near infinity using (\ref{riesz}) as follows
\begin{align} \label{Eq:Asmp_W}
    W_{\alpha,\sigma}(n) & =4^\sigma\frac{\Gamma(\frac{3}{2}-\alpha+\sigma)\Gamma(\alpha)}{\Gamma(\alpha-\sigma)\Gamma(\frac{3}{2}-\alpha)}\frac{\Gamma(n+\alpha-\sigma-1)\Gamma(n-\alpha+2)}{\Gamma(n-\alpha+\sigma+2)\Gamma(n+\alpha-1)} \nonumber \\
    &\sim 4^\sigma\frac{\Gamma(\frac{3}{2}-\alpha+\sigma)\Gamma(\alpha)}{\Gamma(\alpha-\sigma)\Gamma(\frac{3}{2}-\alpha)} n^{-2\sigma} \,,
\end{align}
where the asymptotic behavior follows from \eqref{eq:asympga}.
\end{rmk}

\subsection[Criticality of the quadratic form]{Criticality of $(-\Delta_{\mathbb{N}})^{\sigma}-W_{\alpha,\sigma}$.}
In this section, we will show the criticality of $(-\Delta_{\mathbb{N}})^{\sigma}-W_{\alpha,\sigma}$ for $\sigma \in (0,1]$ and $\alpha \in (\sigma, \frac{3+2\sigma}{4}]$, where $W_{\alpha,\sigma}$ is the Hardy-weight constructed in the previous subsection. For $\sigma \in (0,1)$, we use the criticality characterization (Proposition \ref{Prop:char_cri}) for this purpose. That is, we will construct a null-sequence for $\mathcal{Q}_{W_{\alpha,\sigma}}$ (recall \eqref{eq:quadratic}). On the other hand, for $\sigma =1$, we use the Liouville-type comparison principle (Theorem \ref{Thrm_Liouville}).

While dealing with the energy functional $\mathcal{Q}_{W_{\alpha,\sigma}}$, it is often convenient to consider the following {\it simplified energy functional}:
\begin{equation} \label{Eq:SEF}
Q_{\alpha}^\sigma\left(\varphi\right):=\frac{1}{2} \sum_{n, m \in \mathbb{Z}} (- \widetilde{\mathcal{K}}^\sigma_{m,n}) \mathcal{I}_\alpha(n) \mathcal{I}_\alpha(m)(\varphi(n)-\varphi(m))^2 \,, \ \ \ \varphi \in C_c(\mathbb{N})
\end{equation}
as it is equivalent to  $\mathcal{Q}_{W_{\alpha,\sigma}}$ in the following sense  \cite[Proposition 4.8]{KPP17}
\begin{equation} \label{Eq:GST}
\mathcal{Q}_{W_{\alpha,\sigma}}\left(\varphi\mathcal{I}_\alpha\right)=Q_{\alpha}^\sigma(\varphi)\,, \ \ \ \forall \varphi \in C_c(\mathbb{N}) \,.
\end{equation}
Note that (\ref{Eq:GST}) is just a way to rewrite (\ref{Eq:GSR}) in terms of quadratic forms.
\begin{proposition} \label{Prop:crit}
    Let $\sigma \in (0,1)$. Then $(-\Delta_{\mathbb{N}})^{\sigma} - W_{\alpha,\sigma}$ is critical for all $\alpha \in (\sigma, \frac{3+2\sigma}{4}]$.
\end{proposition}
\begin{proof}
In the view of Proposition \ref{Prop:char_cri}, we aim to construct a null-sequence for $\mathcal{Q}_{W_{\alpha,\sigma}}$.
  For $k\in \mathbb{N}$, we define $\varphi_k:\mathbb{N}\rightarrow\mathbb{R}$,
    \begin{equation*}
       \varphi_k(n)=\left\{\begin{matrix}
           1\\
           2-\frac{\log n}{\log k}\\
           0
       \end{matrix}\right.
       \begin{matrix}
         \ ,\ \\
         \ ,\ \\
          ,
       \end{matrix}
       \begin{matrix}
           n\leq k \,,\\
           k<n\leq k^2 \,,\\
           k^2<n\,.
       \end{matrix}
   \end{equation*}
   This sequence is in $C_c(\mathbb{N})$ and $\varphi_k \overset{k\rightarrow\infty}{\longrightarrow} 1$ pointwise. Furthermore, we claim that
   \begin{equation} \label{Eq:claim}
    (\varphi_k(n)-\varphi_k(m))^2 \leq \left(\frac{\log (m/n)}{\log k}\right)^2.
\end{equation}
Due to symmetry, it is enough to prove the above inequality only for $n<m$. If $k^2\leq n<m$, then the inequality is obvious.
If $n < k^2 < m$, then there can be two situations $(i)$ $n<k$ and $(ii)$ $n \geq k$. In situation $(i)$, we have
\begin{equation*}
    (\varphi_k(n)-\varphi_k(m))^2=1= (2-1)^2 \leq\left(\frac{\log m}{\log k}-\frac{\log n}{\log k}\right)^2=\left(\frac{\log (m/n)}{\log k}\right)^2.
\end{equation*}
For situation $(ii)$
\begin{equation*}
    (\varphi_k(n)-\varphi_k(m))^2=(2-\frac{\log n}{\log k})^2 \leq\left(\frac{\log m}{\log k}-\frac{\log n}{\log k}\right)^2=\left(\frac{\log (m/n)}{\log k}\right)^2.
\end{equation*}
Now, consider $n<m \leq k^2$. Three cases may occur $(i)$ $k \leq n <m \leq k^2$, $(ii)$ $n < k <m \leq k^2$ and $(iii)$ $n<m \leq k$. For case $(i)$, 
\begin{equation*}
    (\varphi_k(n)-\varphi_k(m))^2=\left(\frac{\log n}{\log k}-\frac{\log m}{\log k}\right)^2=\left(\frac{\log (m/n)}{\log k}\right)^2 \,.
\end{equation*}
In case $(ii)$, we have
\begin{equation*}
    (\varphi_k(n)-\varphi_k(m))^2=\left(1-\frac{\log m}{\log k}\right)^2 \leq \left(\frac{\log n}{\log k}-\frac{\log m}{\log k}\right)^2 =\left(\frac{\log (m/n)}{\log k}\right)^2 \,.
\end{equation*}
When case $(iii)$ occurs, \eqref{Eq:claim} holds trivially.

We will use the estimate \eqref{Eq:claim} to show that $Q_{\alpha}^\sigma\left(\varphi_k\right) \rightarrow 0$ as $k \rightarrow \infty$, where $Q_{\alpha}^\sigma$ is the simplified energy functional defined in \eqref{Eq:SEF}.
Using \eqref{Eq:Asmp_K}, we get
\begin{align} \label{Eq:split}
\begin{split}
    Q_{\alpha}^\sigma\left(\varphi_k\right)&\lesssim \sum_{1\leq n <m}\frac{(\varphi_k(n)-\varphi_k(m))^2}{(m-n)^{2\sigma+1}m^{2-2\alpha}n^{2-2\alpha}} \\
    &= \sum_{1\leq n<m\leq k^2}\frac{(\varphi_k(n)-\varphi_k(m))^2}{(m-n)^{2\sigma+1}m^{2-2\alpha}n^{2-2\alpha}}+\sum_{1\leq n\leq k^2<m}\frac{(\varphi_k(n)-\varphi_k(m))^2}{(m-n)^{2\sigma+1}m^{2-2\alpha}n^{2-2\alpha}} \,.
    \end{split}
\end{align}
Let us first deal with the first sum. We use \eqref{Eq:claim} to obtain
\begin{align*}
    &\sum_{1\leq n<m\leq k^2}\frac{(\varphi_k(n)-\varphi_k(m))^2}{(m-n)^{2\sigma+1}m^{2-2\alpha}n^{2-2\alpha}}\leq 
    \frac{1}{(\log k )^2}\sum_{1\leq n<m\leq k^2}\frac{(\log (m/n))^2}{(m-n)^{2\sigma+1}m^{2-2\alpha}n^{2-2\alpha}}\\
    &\leq\frac{1}{(\log k )^2}\sum_{n\leq k^2}\frac{1}{n^{2(2-2\alpha)+1+2\sigma}}\sum_{n<m\leq k^2}\frac{(\log (m/n))^2}{(m/n-1)^{2\sigma+1}(m/n)^{2-2\alpha}}\\
    &\leq\frac{1}{(\log k )^2}\sum_{n\leq k^2}\frac{1}{n^{2(2-2\alpha)+2\sigma}}\int_{\frac{n+1}{n}}^{\frac{k^2}{n}}\frac{(\log t)^2\ \mathrm{d}t}{(t-1)^{2\sigma+1}(t)^{2-2\alpha}} \\
    &\leq\frac{1}{(\log k )^2} \left[\int_{1}^{\infty}\frac{(\log t)^2\ \mathrm{d}t}{(t-1)^{2\sigma+1}(t)^{2-2\alpha}} \right] \left[\sum_{n\leq k^2}\frac{1}{n^{2(2-2\alpha)+2\sigma}} \right]  \\
    &\leq\frac{1}{(\log k )^2} \left[\int_{1}^{\infty}\frac{(\log t)^2\ \mathrm{d}t}{(t-1)^{2\sigma+1}(t)^{2-2\alpha}} \right] \left[\int_1^{k^2} \frac{\mathrm{d}t}{t^{2(2-2\alpha)+2\sigma}} \right]  \,. 
\end{align*} 
For the second sum in \eqref{Eq:split}, we use \eqref{Eq:claim} to obtain
\begin{align*}
    &\sum_{1\leq n\leq k^2<m}\frac{(\varphi_k(n)-\varphi_k(m))^2}{(m-n)^{2\sigma+1}m^{2-2\alpha}n^{2-2\alpha}}\leq 
    \frac{1}{(\log k )^2}\sum_{1\leq n\leq k^2<m}\frac{(\log (m/n))^2}{(m-n)^{2\sigma+1}m^{2-2\alpha}n^{2-2\alpha}}\\
    &\leq\frac{1}{(\log k )^2}\sum_{n\leq k^2}\frac{1}{n^{2(2-2\alpha)+1+2\sigma}}\sum_{k^2<m}\frac{(\log (m/n))^2}{(m/n-1)^{2\sigma+1}(m/n)^{2-2\alpha}}\\
    &\leq\frac{1}{(\log k )^2}\sum_{n\leq k^2}\frac{1}{n^{2(2-2\alpha)+2\sigma}}\int_{\frac{k^2}{n}}^{\infty}\frac{(\log t)^2\ \mathrm{d}t}{(t-1)^{2\sigma+1}(t)^{2-2\alpha}} \\
    &\leq\frac{1}{(\log k )^2} \left[\int_{1}^{\infty}\frac{(\log t)^2\ \mathrm{d}t}{(t-1)^{2\sigma+1}(t)^{2-2\alpha}} \right] \left[\sum_{n\leq k^2}\frac{1}{n^{2(2-2\alpha)+2\sigma}} \right]  \\
    &\leq\frac{1}{(\log k )^2} \left[\int_{1}^{\infty}\frac{(\log t)^2\ \mathrm{d}t}{(t-1)^{2\sigma+1}(t)^{2-2\alpha}} \right] \left[\int_1^{k^2} \frac{\mathrm{d}t}{t^{2(2-2\alpha)+2\sigma}} \right]   \,.
\end{align*} 
Using the preceding two estimates in \eqref{Eq:split}, we conclude that
$$ Q_{\alpha}^\sigma\left(\varphi_k\right) \lesssim \frac{1}{(\log k)^2 } \left[\int_{1}^{\infty}\frac{(\log t)^2\ \mathrm{d}t}{(t-1)^{2\sigma+1}(t)^{2-2\alpha}} \right] \left[\int_1^{k^2} \frac{\mathrm{d}t}{t^{2(2-2\alpha)+2\sigma}} \right]\rightarrow 0 \,,$$
as $k \rightarrow \infty$. Notice that here we use the fact that $\int_{1}^{\infty}\frac{(\log t)^2\ \mathrm{d}t}{(t-1)^{2\sigma+1}(t)^{2-2\alpha}}< \infty $, which is true if $\sigma \in (0,1)$. By \eqref{Eq:GST}, it follows that $\mathcal{Q}_{W_{\alpha,\sigma}}\left(\varphi_k\mathcal{I}_\alpha\right) \rightarrow 0$ as $k \rightarrow \infty$. This implies $(\varphi_k\mathcal{I}_\alpha)$ is a null-sequence for $\mathcal{Q}_{W_{\alpha,\sigma}}$ which converges to $\mathcal{I}_\alpha$ pointwise as $k \rightarrow \infty$. Hence, $(-\Delta_{\mathbb{N}})^{\sigma}-W_{\alpha,\sigma}$ is critical with $\mathcal{I}_\alpha$ as its Agmon ground state.
\end{proof}

\begin{proposition}\label{prop:crit_laplacian}
    $-\Delta_{\mathbb{N}} - W_{\alpha,1}$ is critical for all $\alpha \in (1, \frac{5}{4}]$.
\end{proposition}
\begin{proof}
    From \cite[Theorem 7.3]{KPP18} we know that $W^{\operatorname{KPP}}(n)=2-\sqrt{\frac{n+1}{n}}-\sqrt{\frac{n-1}{n}}$ is an optimal Hardy-weight for the operator $-\Delta_{\mathbb{N}}$. In particular, the Schr\" odinger operator $-\Delta_{\mathbb{N}}-W^{\operatorname{KPP}}$ is critical with corresponding Agmon ground state $v(n)=n^{1/2}$. On the other hand, from Proposition \ref{Prop:HW}, we see that $-\Delta_{\mathbb{N}}-W_{\alpha,1} \geq 0$ and $\mathcal{I}_\alpha$ is a positive solution of $(-\Delta_{\mathbb{N}}-W_{\alpha,1})[\varphi] = 0$ on $\mathbb{N}$. From Remark \ref{Eq:Asymp_Riesz} we know the asymptotic behaviour of $\mathcal{I}_\alpha$, which implies that
    \begin{equation*}
        \mathcal{I}_\alpha(n)\leq C v(n) 
    \end{equation*}
    for all $n\in\mathbb{N}$ and some $C>0$, if $\alpha\leq 5/4$. Thus, taking $\mathcal{H}_{W'}=-\Delta_{\mathbb{N}}-W^{\operatorname{KPP}}$ and $\mathcal{H}_{W}=-\Delta_{\mathbb{N}}-W_{\alpha,1}$ on $\mathbb{N}$, we verify that the hypotheses of Theorem \ref{Thrm_Liouville} are fulfilled. Therefore, $-\Delta_{\mathbb{N}}-W_{\alpha,1}$ is critical and $\mathcal{I}_\alpha$ is its Agmon ground state.
\end{proof}

\subsection[Null criticality of the quadratic form]{Null criticality of $(-\Delta_{\mathbb{N}})^{\sigma}-W_{\alpha,\sigma}$.}
In the earlier subsection, we have seen that the Hardy-weights $W_{\alpha,\sigma}$ are critical Hardy-weights for $(-\Delta_{\mathbb{N}})^{\sigma}$ if $\sigma \in \left(0,1\right]$ and $\alpha \in (\sigma,\frac{3+2\sigma}{4}]$, and the corresponding Agmon ground state of $(-\Delta_{\mathbb{N}})^{\sigma}-W_{\alpha,\sigma}$ is $\mathcal{I}_{\alpha}$. Now we show that, for $\alpha=\frac{3+2\sigma}{4}$, the Hardy-weight $W_{\alpha,\sigma}$ is in fact null-critical.

\begin{proposition}\label{Prop:null-crit}
 Let $\sigma \in \left(0,1\right]$ and $\sigma<\alpha<\min\{3/2,1+\sigma\}$.   The function $\mathcal{I}_\alpha\in \ell^2(\mathbb{N}, W_{\alpha,\sigma})$ if and only if $\alpha<\frac{3+2\sigma}{4}$. In particular, for $\alpha=\frac{3+2\sigma}{4}$, the Hardy-weight $W_{\alpha,\sigma}$ of $(-\Delta_{\mathbb{N}})^{\sigma}$ is null-critical, i.e. $\mathcal{I}_{\alpha} \notin \ell^2(\mathbb{N},W_{\alpha,\sigma})$.
\end{proposition}
\begin{proof}
First, recall that $W_{\alpha, \sigma}$ is a Hardy-weight by Proposition \ref{Prop:HW}. 
Using the asymptotic of $\mathcal{I}_{\alpha}$ and $W_{\alpha,\sigma}$ given by \eqref{Eq:Asmp_I} and \eqref{Eq:Asmp_W} resp., we see that
    \begin{equation*}
        \sum_{n\in\mathbb{N}}\mathcal{I}_\alpha^2(n)W_{\alpha,\sigma}(n)\asymp\sum_{n\in\mathbb{N}}n^{4\alpha-4-2\sigma}<+\infty
    \end{equation*}
    if and only if ${4\alpha-4-2\sigma}<-1\iff\alpha<\frac{3+2\sigma}{4}$.
\end{proof}
\subsection{A necessary condition for criticality.}
We have seen in Proposition \ref{Prop:crit} that $W_{\alpha,\sigma}$ are critical Hardy-weights when $0<\alpha\leq\frac{3+2\sigma}{4}$. We will also show that, in fact, this is also a necessary condition for criticality, i.e. the Hardy-weights $W_{\alpha,\sigma}$ are not critical for $\frac{3+2\sigma}{4}<\alpha<\min\{3/2,1+\sigma\}$. We follow an analogous strategy to the one shown in \cite{KN23}. First, we prove the following result on asymptotic optimality. The proof is completely analogous to Lemma 6 in \cite{KN23} and Theorem 4.3 in \cite{DKPM}. Nevertheless, we state and prove here the result for the sake of reading.

\begin{lemma}{(Optimality near infinity).}\label{Lemma:asympt-crit}
 Let $\sigma \in \left(0,1\right]$ and $\sigma<\alpha<\min\{3/2,1+\sigma\}$. Assume $W_{\alpha,\sigma}$ is a critical Hardy-weight of the form (\ref{weight}). If for a finite set $K\subseteq\mathbb{N}$ and $W':\mathbb{N}\rightarrow\left[0,\infty\right)$
\begin{equation*}
    \langle(-\Delta)^\sigma\varphi,\varphi\rangle\geq  \langle(W_{\alpha,\sigma}+W')\varphi,\varphi\rangle
\end{equation*}
for all $\varphi\in C_c(\mathbb{N}\setminus K)$, then $\mathcal{I}_\alpha\in\ell^2(\mathbb{N},W')$.

In particular, if we choose $W'=\lambda W_{\alpha,\sigma}$, we see that the weight is optimal near infinity if it is null-critical.
\end{lemma}

\begin{proof}
    The assumption gives us
    \begin{equation*}
        \langle W'\varphi,\varphi\rangle\leq \langle((-\Delta)^\sigma-W_{\alpha,\sigma})\varphi,\varphi\rangle=\left(Q^\sigma-W_{\sigma, \alpha}\right)\left(\varphi\right).
    \end{equation*}
    Let $\varphi=\mathcal{I}_\alpha\eta1_{\mathbb{N}\setminus K}$ with $\eta\in C_c(\mathbb{N})$ and $0\leq\eta\leq1$. By the ground state transform
    \begin{equation*}
        \sum_{n\in\mathbb{N}\setminus K}W'(n)(\mathcal{I}_\alpha\eta)^2(n)\leq \left(Q^\sigma-W_{\sigma, \alpha}\right)\left(\mathcal{I}_\alpha\eta1_{\mathbb{N}\setminus K}\right)=Q_{-\alpha}^\sigma(\eta(1-1_K))\leq 2Q_{-\alpha}^\sigma(\eta)+2Q_{-\alpha}^\sigma(1_K)
    \end{equation*}
    Now we replace $\eta$ by a null-sequence which exists since we assumed criticality (Proposition \ref{Prop:char_cri}). Then by Fatou's lemma we get
    \begin{equation*}
         \sum_{n\in\mathbb{N}\setminus K}W'(n)(\mathcal{I}_\alpha)^2(n)<\infty.
    \end{equation*}
    The statement about null-criticality comes directly by choosing $W'=\lambda W_{\alpha,\sigma}$
\end{proof}

Now we see that, if we assume that $W_{\alpha,\sigma}$ are critical for $\frac{3+2\sigma}{2}<\alpha<\min\{3/2,1+\sigma\}$, then Lemma \ref{Lemma:asympt-crit} leads to a contradiction of the null criticality described in Proposition \ref{Prop:null-crit}.

\begin{proposition}\label{Prop:necess}
    Let $0<\sigma\leq1$, $\sigma<\alpha<1+\sigma$, and $\frac{3+2\sigma}{4}<\alpha<3/2$. Then the Hardy-weights $W_{\alpha,\sigma}$ are not critical.
\end{proposition}

\begin{proof}
    As can be observed from (\ref{Eq:Asmp_W}), all weights $W_{\alpha,\sigma}$ share the same asymptotic decay, $n^{-2\sigma}$, but differ in the constants, which are given by 
    \begin{equation*}
        \Psi_\sigma(\alpha)=4^\sigma\frac{\Gamma(\frac{3}{2}-\alpha+\sigma)\Gamma(\alpha)}{\Gamma(\alpha-\sigma)\Gamma(\frac{3}{2}-\alpha)}.
    \end{equation*}

    To study the monotonicity of $\Psi_\sigma(\alpha)$ with respect to $\alpha$, we compute $\Psi'_\sigma(\alpha)$

    \begin{equation*}
        \Psi'_\sigma(\alpha)=4^\sigma\frac{\Gamma(\frac{3}{2}-\alpha+\sigma)\Gamma(\alpha)}{\Gamma(\alpha-\sigma)\Gamma(\frac{3}{2}-\alpha)}[\psi^{(0)}(\alpha)-\psi^{(0)}(\alpha-\sigma)-\psi^{(0)}(\frac{3}{2}-\alpha+\sigma)+\psi^{(0)}(\frac{3}{2}-\alpha)].
    \end{equation*}    
    using $\Gamma'(x)=\Gamma(x)\psi^{(0)}(x)$. We can write the terms inside of the brackets as
    \begin{equation*}
        \int_{\alpha-\sigma}^{\alpha}[\psi^{(1)}(t)+\psi^{(1)}(3/2-t)]\ \mathrm{d}t= \int_{\alpha-\sigma}^{\alpha}\psi^{(1)}(t)\ \mathrm{d}t - \int_{3/2-\alpha}^{3/2-\alpha+\sigma}\psi^{(1)}(t)\ \mathrm{d}t.
    \end{equation*}
    Since $\psi^{(1)}$ is a strictly decreasing positive function on the positive real axis, both intervals of integration are of measure $\sigma$ and by assumption $\alpha-\sigma>3/2-\alpha$, then $ \Psi'_\sigma(\alpha)<0$ in $(\frac{3+2\sigma}{4},\frac{3}{2})$. Therefore, $\Psi_\sigma(\alpha)$ is strictly decreasing in that interval and $\Psi_{\sigma,3}(\frac{3+2\sigma}{4})=4^{\sigma}\frac{\Gamma\left(\frac{3+2\sigma}{4}\right)^2}{\Gamma\left(\frac{3-2\sigma}{4}\right)^2}=:C_\sigma$.    
    Now we argue by contradiction. Assume $W_{\alpha,\sigma}$ is critical for some $\alpha>\frac{3+2\sigma}{4}$. By the above analysis of $ \Psi_\sigma$ we can choose a positive $\lambda$ such that
    \begin{equation*}
        \frac{C_\sigma - \Psi_\sigma(\alpha)}{\Psi_\sigma(\alpha)}>\lambda\implies C_\sigma>(1+\lambda)\Psi_\sigma(\alpha).
    \end{equation*}
    As the asymptotic behavior of the weights differs only in their constants, we can choose a finite set $K\subseteq\mathbb{N}$ such that outside of $K$ we get
    \begin{equation*}
        W_{\sigma,\frac{3+2\sigma}{4}}\geq (1+\lambda)W_{\alpha,\sigma}
    \end{equation*}
    Since $\mathcal{I}_{\alpha}$ is the unique ground state of $(Q^\sigma-W_{\alpha,\sigma})$, the weight $W'=\lambda W_{\alpha,\sigma}$ fulfills the hypotheses of Lemma \ref{Lemma:asympt-crit}, so we conclude 
    \begin{equation*}
        \mathcal{I}_{\alpha}\in\ell^2(\mathbb{N},W_{\alpha,\sigma}).
    \end{equation*}
    However this contradicts Proposition \ref{Prop:null-crit}.
\end{proof}
\subsection{Proof of Theorem \ref{Theorem:Hardy}.}
We conclude the on-going section with the proof of Theorem \ref{Theorem:Hardy}.

\begin{proof}
Proposition \ref{Prop:HW} shows that $W_{\alpha,\sigma}$ is a Hardy-weight for $(-\Delta)^{\sigma}$ if $\sigma \in \left(0,1\right]$, $0<\alpha-\sigma<\frac{3}{2} $ and $\alpha\in(\sigma,1+\sigma)$. Propositions \ref{Prop:crit} and \ref{Prop:necess} ensure that $(-\Delta)^{\sigma}-W_{\alpha,\sigma}$ is critical if and only if $\alpha \in \left(\sigma,\frac{3+2\sigma}{4}\right]$. Finally, Proposition \ref{Prop:null-crit}, implies that the Hardy-weight $W_{\alpha,\sigma}$ of $(-\Delta)^{\sigma}$ is null-critical if $\alpha=\frac{3+2\sigma}{4}$. Hence,  we prove the optimality of the Hardy-weight $W_{\frac{3+2\sigma}{4},\sigma}$ for $(-\Delta)^{\sigma}$.    
\end{proof}

\section{An application to unique continuation at infinity} \label{Sec:UCP}
As an immediate application to the above study of optimal Hardy-weights, here we derive some unique continuation results as stated in Theorem \ref{Thm:UCP}.

We define $\mathcal{G}_{\sigma}= \sqrt{2/\pi}  \mathcal{I}_{\sigma}$. 
Then, for $\sigma \in {(0,\frac{3}{2})}$, \begin{align*} 
    \mathcal{G}_{\sigma} (n) &={\sqrt{\frac{2}{\pi}} }\langle \mathcal{U} e_n,[2^{-\sigma}[1-\cdot]^{-\sigma}]\rangle_{L^2((-1,1),\sqrt{1-x^2}\mathrm{d} x)} \\
  & = \frac{2}{\pi} \int_{-1}^1 \frac{U_{n-1}}{2^{\sigma}[1-x]^{\sigma}} \sqrt{1-x^2} \ dx \\
  & = \frac{2}{\pi} \lim_{\lambda \rightarrow 0_+} \int_{-1}^1 \frac{U_{n-1}}{2^{\sigma}[1-x+\lambda]^{\sigma}} \sqrt{1-x^2} \ dx =  \lim_{\lambda \rightarrow 0_+} [(-\Delta_{\mathbb{N}})^{\sigma}+\lambda]^{-1}_{1,n} \,,
\end{align*}
where the third equality uses the dominated convergence theorem and the fourth one follows from \cite[Corollary 2.2]{GKS25}. Restricting $\sigma \in (0,1]$, it follows from \cite[Theorem 6.26]{KLW21} that $(-\Delta_{\mathbb{N}})^{\sigma} \mathcal{G}_{\sigma} = \delta_{1,\cdot}$ and $\mathcal{G}_{\sigma}$ is  the smallest positive $u \in \mathcal{F}_{\mathbb{N}}^{\sigma}$ such that  $(-\Delta_{\mathbb{N}})^{\sigma}u \geq  \delta_{1,\cdot}$ on $\mathbb{N}$.  Hence, $\mathcal{G}_{\sigma}$ is a positive  minimal Green function of $(-\Delta_{\mathbb{N}})^{\sigma}$, i.e., for any $v>0$ with  $(-\Delta_{\mathbb{N}})^{\sigma}v \geq  0$ outside a finite set in $\mathbb{N}$, there exists $C>0$ such that $\mathcal{G}_{\sigma} \leq C v$ on $\mathbb{N}$.

\begin{proof}[Proof of Theorem \ref{Thm:UCP}]
The proof follows from \cite[Theorem 3.1]{DKP24}. We give the details for reader's convenience. Let $ u $ be a solution of $\left[ (-\Delta_{\mathbb{N}})^{\sigma} + V\right] \varphi =0$ on $\mathbb{N}$, and assume without loss of generality $ u_{+}\neq 0 $ (otherwise consider $ -u$). Then, by \cite[Lemma 3.2]{DKP24}, we have
\begin{align*}
	\left[ (-\Delta_{\mathbb{N}})^{\sigma} + V\right] u_{+} \le 0 
\end{align*}
on $\mathbb{N}$. Now, in the Liouville comparison principle (Theorem \ref{Thrm_Liouville}), we take $\mathcal{H}_{W'}= (-\Delta_{\mathbb{N}})^{\sigma} - W^{\mathrm{op}}_{\sigma}$ and $\mathcal{H}_{W}= (-\Delta_{\mathbb{N}})^{\sigma} +V$, where $W^{\mathrm{op}}_{\sigma}$ is the optimal Hardy-weight of $(-\Delta_{\mathbb{N}})^{\sigma}$ given in Theorem \ref{Theorem:Hardy}. Then, due to the first decay assumption on $u$, it is not difficult to verify that the hypothesis of Theorem \ref{Thrm_Liouville} are satisfied. Hence, we can infer that $\mathcal{H}_{W}= (-\Delta_{\mathbb{N}})^{\sigma} +V$ is in fact critical and $u=u_+>0$ is the corresponding Agmon ground state. Since $u>0$ and $V \leq 0$, it follows that
$$(-\Delta_{\mathbb{N}})^{\sigma} u \geq \left[(-\Delta_{\mathbb{N}})^{\sigma} +V \right] u =0 $$
outside a finite set in $\mathbb{N}$. Since $\mathcal{G}_{\sigma}$ is a positive minimal Green function of $(-\Delta_{\mathbb{N}})^{\sigma}$, it follows that there exists $C>0$ such that $\mathcal{G}_{\sigma} \leq C u$ on $\mathbb{N}$. As $\mathcal{G}_{\sigma}$ has the asymptotics $\mathcal{G}_{\sigma}(n) \asymp n^{2\sigma -2}$ for large $n$, we get a contradiction to our second decay assumption. Therefore, $u \equiv 0$.  
\end{proof}

\begin{rmk}
\rm There exists $C>0$ such that $(-\Delta_{\mathbb{N}})^{\sigma} - C \delta_{1,\cdot}$ is critical in $\mathbb{N}$ \cite[Lemma 4.4.]{DKPM}. Let $u$ be its Agmon ground state. In particular, $\left[(-\Delta_{\mathbb{N}})^{\sigma} - C \delta_{1,\cdot}\right] u =0$ in $\mathbb{N}$. Moreover, it is not difficult to see that $u \asymp \mathcal{G}_{\sigma}$ outside a finite set (as both are positive solutions of minimal growth at infinity of the same equation). This shows that the decay assumption in Theorem \ref{Thm:UCP} is sharp.   
\end{rmk}

\section{Final Remarks}\label{Sec:Final-Rem}

In Theorem \ref{Theorem:Hardy}, we provide an optimal Hardy-weight for the fractional Laplacian $(-\Delta_{\mathbb{N}})^{\sigma}$ over $\mathbb{N}$ for the exponents $\sigma\in\left(0,1\right]$. As a simple consequence (using Theorem \ref{Theorem:Hardy} and \eqref{eq:asympga}), we obtain the largest value of $\gamma$ such that \eqref{Eq:HW} describes a Hardy-weight for $(-\Delta_\mathbb{N})^\sigma$. From Theorem \ref{Theorem:Hardy} and Remark \ref{Rmk:weight-asympt}, we have
$$W^{\mathrm{op}}_{\sigma} \sim 4^{\sigma}\frac{\Gamma\left(\frac{3+2\sigma}{4}\right)^2}{\Gamma\left(\frac{3-2\sigma}{4}\right)^2} n^{-2\sigma} \,.$$
Further, the following is also true
\begin{lemma}\label{lemm_ct}
    Let $n\in\mathbb{N}$, $0<\sigma<1$. Then
    \begin{equation}
        W(n)=\frac{\Gamma(n-\frac{1+2\sigma}{4})\Gamma(n+\frac{5-2\sigma}{4})}{\Gamma(n+\frac{-1+2\sigma}{4})\Gamma(n+\frac{5+2\sigma}{4})}\geq n^{-2\sigma}.
    \end{equation}
\end{lemma}
We write the proof in Appendix \ref{app_ct}. This lemma implies that $\gamma_H:=4^{\sigma}\frac{\Gamma\left(\frac{3+2\sigma}{4}\right)^2}{\Gamma\left(\frac{3-2\sigma}{4}\right)^2}$ is the largest possible $\gamma$ such that the Hardy inequality $(-\Delta_{\mathbb{N}})^{\sigma}-\gamma |n|^{-2\sigma} \geq 0$ holds on $\mathbb{N}$.

In this context, it is relevant to also point out that the best Hardy constant $\gamma_H$ coincides with the best Hardy constant at infinity, which is defined by $$ \gamma_{\rm{H}}^{\infty}= \sup_{K \Subset \mathbb{N}}\inf \left \{\frac{(-\Delta_{\mathbb{N}})^{\sigma}  f,f\rangle_{\ell^2(\mathbb{N})}}{\langle f,f\rangle_{\ell^2(\mathbb{N},|n|^{-2\sigma})}}\,, f \in C_c(\mathbb{N}\setminus K)  \right\} \,.
$$
 In fact, $\gamma_{\rm{H}}^{\infty}$ is the bottom of the (essential) spectrum of the weighted operator $|n|^{2\sigma}(-\Delta_{\mathbb{N}})^{\sigma}$ in $\ell^2(\mathbb{N},|n|^{-2\sigma})$. The constant $\gamma_{\rm{H}}$ plays a significant role in the analysis of Hardy-type inequalities; for instance, see \cite{DDP} and the references therein.

Another interesting remark is that the optimal weight $W_{1}^{\text{op}}$ we find in Theorem \ref{Theorem:Hardy} for $\sigma=1$ is different from the one obtained in \cite[Theorem 7.3]{KPP18}, and the ones obtained in \cite[Corollary 12]{KLS}, all of them are also optimal. Recall that
\begin{equation}
    W^{\operatorname{KPP}}(n)=2-\sqrt{1+\frac{1}{n}}-\sqrt{1-\frac{1}{n}}=\frac{1}{4n^2}+\frac{5}{64n^4}+O(n^{-6})
\end{equation}
\begin{equation}
    W_{1}^{\text{op}}(n)=\frac{1}{4}\frac{1}{n^2-\frac{9}{16}}=\frac{1}{4n^2}+\frac{9}{64n^4}+O(n^{-6})\,.
\end{equation}
Note that the weights are not comparable as $W^{\operatorname{KPP}}(1)>W_{1}^{\text{op}}(1)$ and $W_{1}^{\text{op}}(n)> W^{\operatorname{KPP}}(n)$\footnote{ Of course, this implies that $W^{\mathrm{op}}_1$ and $W^{\mathrm{KPP}}$ must intersect each other if we consider them as functions on $\mathbb{R}_+$. It can be checked with the help of a symbolic computation software that the intersecting point is in between 1 and 2.} for large $n$. This fact highlights the non-uniqueness of optimal Hardy-weights, and raises the question of finding the complete class of optimal Hardy-weights.
Even though $ W_{1}^{\text{op}}$ is asymptotically larger than $W^{\operatorname{KPP}}$ at order $n^{-4}$, it does not conflict with the fact that $W^{\operatorname{KPP}}$ is optimal at infinity as well, since Remark \ref{Rmk:Asympt-Optimal} remains true for both weights.

Finally, it is surprising that the optimal constant $C_{\sigma}$ appearing in the discrete fractional Hardy inequality on the discrete half-line does not coincide with its continuum counterpart obtained in \cite{BD}: 
\begin{equation*}
    \gamma_{\operatorname{BD}}=\frac{2\Gamma(1-\sigma)\Gamma(\sigma+1/2)-4^\sigma}{4^\sigma\sigma \sqrt{\pi}}.
\end{equation*}
This contrasts with the full-line case, where the best constant in the Hardy inequality for discrete fractional Laplacian is equal to the one corresponding to the continuum case (see \cite{KN23} and \cite{Herbst} respectively).

{\textbf{Open Question.}}  We  know from \cite{GKS25} that $(-\Delta_\mathbb{N})^\sigma$ is subcritical for $\sigma\in(0,3/2)$, but we obtain optimal Hardy-weights for $\sigma\in(0,1]$.  So the next natural question is to find an optimal Hardy-weight for the remaining exponents in the range of subcriticality. The main obstacle we face when using our approach is that the operator no longer behaves as a graph Laplacian, since the kernel $-\widetilde{\mathcal{K}}^\sigma_{m,n}$ is no longer nonnegative for $\sigma>1$. 
A different strategy is thus needed in order to tackle this problem for $\sigma\in(1,3/2)$.

\appendix
\section{Computations}{\label{comp}}
This appendix is devoted to proving several facts about Gamma functions and binomial coefficients to aid us in our computations. See \cite{AR} for general properties of Gamma functions. 
\begin{lemma}
    \label{LR18-comp}
    Let $\sigma\in\mathbb{R}\setminus\mathbb{Z}$ and $n\in\mathbb{N}$. Then: 
    \begin{align}
        \mathcal{K}^\sigma_{m,n}&=(-1)\left(\frac{(-1)^{m-n+1}\Gamma(2\sigma+1)}{\Gamma(\sigma+m-n+1)\Gamma(\sigma-m+n+1)}-\frac{(-1)^{m+n+1}\Gamma(2\sigma+1)}{\Gamma(\sigma+m+n+1)\Gamma(\sigma-m-n+1)}\right)\\
        &=\frac{(-1)\Gamma(\sigma + 1/2)4^\sigma}{\sqrt{\pi}\Gamma(-\sigma)}\left(\frac{\Gamma(m+n-\sigma)}{\Gamma(m+n+\sigma+1)}-\frac{\Gamma(m-n-\sigma)}{\Gamma(m-n+\sigma+1)}\right)
    \end{align}
\end{lemma}
\begin{proof}
    Define the functions $F(n)$ and $H(n)$ as
    \begin{equation*}
        F(n):=\frac{(-1)^{n}\Gamma(2\sigma+1)}{\Gamma(\sigma+n+1)\Gamma(\sigma-n+1)}\ \ ; \ \ H(n):= \frac{(-1)\Gamma(\sigma + 1/2)4^\sigma}{\sqrt{\pi}\Gamma(-\sigma)}\frac{\Gamma(n-\sigma)}{\Gamma(n+\sigma+1)}.
    \end{equation*}
    By \cite[Proposition 1]{LR18} we have $-F(n)=H(n)$. Note that, in their notation, $F\equiv K^\sigma$ and $H\equiv R^\sigma$. Also note that although \cite[Proposition 1]{LR18} is only proved for $\sigma\in (0,1)$ it can easily be extended to $\mathbb{R}\setminus\mathbb{Z}$ by not using $\Gamma(-\sigma)=-|\Gamma(-\sigma)|$. The rest of their proof only uses Legendre's duplication formula and Euler’s reflection formula.
    
    Without loss of generality, consider $m>n$. We can write 
    \begin{equation*}
         \mathcal{K}^\sigma_{m,n}=F(m-n)-F(m+n)=H(m+n)-H(m-n),
    \end{equation*}
    which completes the proof
\end{proof}

\begin{lemma}
\label{sum1}
    Let $\sigma\in\mathbb{R}$ and $n\in\mathbb{N}$. Then:
    \begin{equation}
        \sum_{m=0}^{n-1}(-1)^{m}\left(\begin{matrix} 2\sigma\\\sigma+m\end{matrix}\right)=\frac{\Gamma(2\sigma)}{\Gamma(\sigma)\Gamma(\sigma+1)}+(-1)^{n-1}\frac{\Gamma(2\sigma)}{\Gamma(\sigma+n)\Gamma(\sigma-n+1)}
    \end{equation}
\end{lemma}
\begin{proof}
    By induction, for $n=1$:
    \begin{equation*}
        \left(\begin{matrix} 2\sigma\\\sigma\end{matrix}\right)=\frac{\Gamma(2\sigma+1)}{\Gamma(\sigma+1)\Gamma(\sigma+1)}=\frac{2 \sigma \Gamma(2\sigma)}{\sigma \Gamma(\sigma)\Gamma(\sigma+1)}=\frac{2  \Gamma(2\sigma)}{\Gamma(\sigma)\Gamma(\sigma+1)}.
    \end{equation*}
    Now assume
    \begin{equation*}
         \sum_{m=0}^{n-2}(-1)^{m}\left(\begin{matrix} 2\sigma\\\sigma+m\end{matrix}\right)=\frac{\Gamma(2\sigma)}{\Gamma(\sigma)\Gamma(\sigma+1)}+(-1)^{n-2}\frac{\Gamma(2\sigma)}{\Gamma(\sigma-n+2)\Gamma(\sigma+n-1)},
    \end{equation*}
    then
    \begin{align*}
         &\sum_{m=0}^{n-1}(-1)^{m}\left(\begin{matrix} 2\sigma\\\sigma+m\end{matrix}\right)\\
         &=\frac{\Gamma(2\sigma)}{\Gamma(\sigma)\Gamma(\sigma+1)}+(-1)^{n-2}\frac{\Gamma(2\sigma)}{\Gamma(\sigma-n+2)\Gamma(\sigma+n-1)}+(-1)^{n-1}\frac{\Gamma(2\sigma+1)}{\Gamma(\sigma+n)\Gamma(\sigma-n+2)}\\
         &=\frac{\Gamma(2\sigma)}{\Gamma(\sigma)\Gamma(\sigma+1)}+(-1)^{n-2}\frac{\Gamma(2\sigma)(\sigma+n-1)}{\Gamma(\sigma-n+2)\Gamma(\sigma+n)}+(-1)^{n-1}\frac{2\sigma\Gamma(2\sigma)}{\Gamma(\sigma+n)\Gamma(\sigma-n+2)}\\
         &=\frac{\Gamma(2\sigma)}{\Gamma(\sigma)\Gamma(\sigma+1)}+(-1)^{n-1}\frac{\Gamma(2\sigma)}{\Gamma(\sigma+n)\Gamma(\sigma-n+2)}\left(\sigma-n+1\right)\\
         &=\frac{\Gamma(2\sigma)}{\Gamma(\sigma)\Gamma(\sigma+1)}+(-1)^{n-1}\frac{\Gamma(2\sigma)}{\Gamma(\sigma+n)\Gamma(\sigma-n+1)},
    \end{align*}
    as desired.
\end{proof}

\begin{lemma}
\label{sum2}
    Let $\sigma\in\mathbb{R}$ and $n\in\mathbb{N}$. Then:
    \begin{equation}
        \sum_{m=0}^{n-1}(-1)^{m+1}\left(\begin{matrix} 2\sigma\\\sigma+m+1\end{matrix}\right)=-\frac{\Gamma(2\sigma)}{\Gamma(\sigma)\Gamma(\sigma+1)}+(-1)^{n}\frac{\Gamma(2\sigma)}{\Gamma(\sigma-n)\Gamma(\sigma+n+1)}.
    \end{equation}
\end{lemma}
\begin{proof}
    Repeat the computation for Lemma \ref{sum1}.
\end{proof}
We also have the following.

\begin{lemma}\label{potential-sum}
   Let $\sigma\in\mathbb{R}$ and $n\in\mathbb{N}$. Then:
\begin{equation}
\label{eq:A3}
     \sum_{m=-n+1}^{n} (-1)^{m} \left(\begin{matrix}
 2\sigma\\
\sigma+m
\end{matrix} 	
 \right)= -2 (-1)^n n \frac{\Gamma(2\sigma)}{\Gamma(1+\sigma-n)\Gamma(1+\sigma+n)}.
\end{equation}
\end{lemma}
\begin{proof} 
We expand the left hand side of \eqref{eq:A3} to obtain
    \begin{align*}
    &\sum_{m=-n+1}^{n} (-1)^{m} \left(\begin{matrix}
 2\sigma\\
\sigma+m
\end{matrix} 	
 \right) =\sum_{m=0}^{n-1} (-1)^{m+1} \left(\begin{matrix}
 2\sigma\\
\sigma+m+1
\end{matrix} 	
 \right)+\sum_{m=0}^{n-1} (-1)^{m} \left(\begin{matrix}
 2\sigma\\
\sigma+m
\end{matrix} 	
 \right)\\
 &\quad= (-1)^{n}\left(\frac{\Gamma(2\sigma)}{\Gamma(\sigma-n)\Gamma(\sigma+n+1)}-\frac{\Gamma(2\sigma)}{\Gamma(\sigma+n)\Gamma(\sigma-n+1)}\right)\\
 &\quad= (-1)^{n}\Gamma(2\sigma)\frac{\Gamma(\sigma+n)\Gamma(\sigma-n+1)-\Gamma(\sigma-n)\Gamma(\sigma+n+1)}{\Gamma(\sigma-n)\Gamma(\sigma+n+1)\Gamma(\sigma+n)\Gamma(\sigma-n+1)}\\
 &\quad=(-1)^{n}\Gamma(2\sigma)\frac{(\sigma-n)-(\sigma+n)}{\Gamma(\sigma+n+1)\Gamma(\sigma-n+1)}=-2(-1)^{n}n\frac{\Gamma(2\sigma)}{\Gamma(\sigma+n+1)\Gamma(\sigma-n+1)}
\end{align*}
where we used Lemmas \ref{sum1} and \ref{sum2} in the second equality.
\end{proof}

\begin{lemma}\label{simpriesz}
     Let $\alpha\in\mathbb{R}$ and $n\in\mathbb{N}$. Then:
     \begin{equation*}
          \left[\frac{\Gamma(\alpha-n-1)}{\Gamma(-n-\alpha)}-\frac{\Gamma(\alpha+n-1)}{\Gamma(n-\alpha)}\right]=\frac{\Gamma(n+\alpha-1)}{\Gamma(n-\alpha+2)}2n(2\alpha-1).
     \end{equation*}
\end{lemma}
\begin{proof}
    We start with the first term
    \begin{equation*}
        \frac{\Gamma(\alpha-n-1)}{\Gamma(-n-\alpha)}=\frac{\sin(\pi(\alpha+n+1))\Gamma(\alpha+n+1)}{\sin(\pi(2+n-\alpha))\Gamma(2+n-\alpha)}=\frac{(-1)^{n+1}}{(-1)^{-n-1}}\frac{\sin\pi\alpha}{\sin\pi\alpha}\frac{\Gamma(\alpha+n+1)}{\Gamma(2+n-\alpha)}.
    \end{equation*}
    Now we put everything together
    \begin{align*}
         &\left[\frac{\Gamma(\alpha-n-1)}{\Gamma(-n-\alpha)}-\frac{\Gamma(\alpha+n-1)}{\Gamma(n-\alpha)}\right]= \left[\frac{\Gamma(\alpha+n+1)}{\Gamma(2+n-\alpha)}-\frac{\Gamma(\alpha+n-1)}{\Gamma(n-\alpha)}\right]\\
         &=\frac{\Gamma(\alpha+n+1)\Gamma(n-\alpha)-\Gamma(\alpha+n-1)\Gamma(2+n-\alpha)}{\Gamma(2+n-\alpha)\Gamma(n-\alpha)}\\
         &=\frac{(n+\alpha)(n+\alpha-1)\Gamma(\alpha+n-1)\Gamma(n-\alpha)-(n-\alpha+1)(n-\alpha)\Gamma(\alpha+n-1)\Gamma(n-\alpha)}{\Gamma(2+n-\alpha)\Gamma(n-\alpha)}\\
         &=\frac{\Gamma(\alpha+n-1)}{\Gamma(2+n-\alpha)}\left[(n+\alpha)(n+\alpha-1)-(n-\alpha+1)(n-\alpha)\right]=\frac{\Gamma(n+\alpha-1)}{\Gamma(n-\alpha+2)}2n(2\alpha-1),
    \end{align*}
    as desired.
\end{proof}

\section{Proof of Lemma \ref{lemm_ct}}\label{app_ct}

We are going to use Theorem 2 from \cite{QG08}:

\begin{theorem}[{\cite[Theorem 2]{QG08}}]\label{QiGuo}
    Let $a$, $b$, and $c$ be real numbers, let $\rho=\min\{a,b,c\}$, and let $\delta>-\rho$.
Then the inequalities
\begin{equation}\label{ineq1}
  (x+c)^{a-b} < \frac{\Gamma(x+a)}{\Gamma(x+b)}
\end{equation}
for $x\in(-\rho,\infty)$ and
\begin{equation}\label{ineq2}
  \frac{\Gamma(x+a)}{\Gamma(x+b)}
  \le \frac{\Gamma(\delta+a)}{\Gamma(\delta+b)}\left(\frac{x+c}{\delta+c}\right)^{a-b}
\end{equation}
for $x\in[\delta,\infty)$ hold if and only if $(a,b,c)\in D_1(a,b,c)$.
The reversed inequalities hold (in the same ranges) if and only if $(a,b,c)\in D_2(a,b,c)$.
\end{theorem}
Here $D_1(a,b,c)$ and $D_2(a,b,c)$ are defined by
\[
\begin{aligned}
D_1(a,b,c):=\Bigl\{(a,b,c)\in\mathbb{R}^3:\;&(b-a)\bigl(1-a-b+2c\bigr)\ge 0,\\
& (b-a)\bigl(|a-b|-a-b+2c\bigr)\ge 0\Bigr\},\\
D_2(a,b,c):=\Bigl\{(a,b,c)\in\mathbb{R}^3:\;&(b-a)\bigl(1-a-b+2c\bigr)\le 0,\\
& (b-a)\bigl(|a-b|-a-b+2c\bigr)\le 0\Bigr\}.
\end{aligned}
\]

We will deal separately with the two quotients of gamma functions.
\begin{lemma}\label{lemma1}
Let $n\in\mathbb{N}$, $0<\sigma<1$. Then
    \begin{equation}
        \frac{\Gamma(n-\frac{1+2\sigma}{4})}{\Gamma(n+\frac{-1+2\sigma}{4})}n^{\sigma}>1 
    \end{equation}
\end{lemma}
\begin{proof}
    We are going to apply (\ref{ineq1}) using $a=-\frac{1}{4}-\frac{\sigma}{2}$, $b=-\frac{1}{4}+\frac{\sigma}{2}$, $c=0$.
    Then 
    $$(b-a)\bigl(1-a-b+2c\bigr)=\frac{3\sigma}{2}>0$$
    and
    $$(b-a)\bigl(|a-b|-a-b+2c\bigr)=\sigma\left(\frac{3}{2}+\sigma\right)> 0\ .$$
    Therefore, $(a,b,c)\in D_1(a,b,c)$ and Theorem \ref{QiGuo} yields the result.
\end{proof}

In order to deal with the second quotient, we need to be more careful, since if we follow the same strategy we find that $(a,b,c)\notin D_1(a,b,c)$. 

\begin{lemma}\label{lemma2}
    Let $n\in\mathbb{N}$, $0<\sigma<1$. Then 
    \begin{equation}
        \frac{\Gamma(n+\frac{5-2\sigma}{4})}{\Gamma(n+\frac{5+2\sigma}{4})}\geq \, n^{-\sigma}.
    \end{equation}
    
\end{lemma}
\begin{proof}
    In this case, we want to use (\ref{ineq2}). More specifically, the reciprocal $$\frac{\Gamma(x+b)}{\Gamma(x+a)}
  \ge \frac{\Gamma(\delta+b)}{\Gamma(\delta+a)}\left(\frac{x+c}{\delta+c}\right)^{b-a}\ .$$
  Let $a=\frac{5}{4}+\frac{\sigma}{2}$, $b=\frac{5}{4}-\frac{\sigma}{2}$, $c=0$. Then 
    $$(b-a)\bigl(1-a-b+2c\bigr)=\frac{3\sigma}{2}>0$$
    and
    $$(b-a)\bigl(|a-b|-a-b+2c\bigr)=\sigma\left(\frac{3}{2}-\sigma\right)> 0\ .$$
    Therefore $(a,b,c)\in D_1(a,b,c)$ and we can use Theorem \ref{QiGuo} to conclude

    \begin{equation*}
        \frac{\Gamma(n+\frac{5-2\sigma}{4})}{\Gamma(n+\frac{5+2\sigma}{4})}\geq \frac{\Gamma(\delta+\frac{5-2\sigma}{4})}{\Gamma(\delta+\frac{5+2\sigma}{4})}\delta^\sigma n^{-\sigma}=C_{\delta,\sigma}n^{-\sigma}
    \end{equation*}
    for $\delta>-\rho=0$. Since $(\frac{5-2\sigma}{4},\frac{5+2\sigma}{4},0)\in D_2(\frac{5-2\sigma}{4},\frac{5+2\sigma}{4},0)$, we use Theorem \ref{QiGuo} again to conclude
    \begin{equation*}
        \frac{\Gamma(\delta+\frac{5-2\sigma}{4})}{\Gamma(\delta+\frac{5+2\sigma}{4})}<\delta^{-\sigma} \implies C_{\delta,\sigma}<1.
    \end{equation*}
    However,  we can choose $\delta$ as large as we want, as Theorem \ref{QiGuo} holds for $\delta\in\left[-\rho,\infty\right)$. Since we also know that 
    \begin{equation*}
        \lim_{\delta\rightarrow\infty}\frac{\Gamma(\delta+\frac{5-2\sigma}{4})}{\Gamma(\delta+\frac{5+2\sigma}{4})}\delta^\sigma=1,
    \end{equation*}
    we can conclude the lemma.
\end{proof}

\begin{proof}[Proof of Lemma \ref{lemm_ct}]
   We now just have to combine lemmas \ref{lemma1} and \ref{lemma2} to get 
    \begin{equation*}
        \frac{\Gamma(n-\frac{1+2\sigma}{4})\Gamma(n+\frac{5-2\sigma}{4})}{\Gamma(n+\frac{-1+2\sigma}{4})\Gamma(n+\frac{5+2\sigma}{4})}n^{2\sigma}>n^{-\sigma}n^{-\sigma}n^{2\sigma}=1
    \end{equation*}
\end{proof}

\addcontentsline{toc}{section}{Acknowledgements}

\section*{Acknowledgements}
The authors sincerely thank Luz Roncal for bringing the problem to our attention. We are thankful to her for several discussions and for her important suggestions to improve the article. 
 We also wish to thank Matthias Keller, David Krej\v{c}i\v{r}ík for some discussions, and for their comments and remarks. We thank Yehuda Pinchover for a careful reading and for his suggestions. We also thank the anonymous referees for their thorough review and valuable comments, which helped us to improve the manuscript.

\section*{Data availability statement}
 The authors declare that the data supporting the findings of this study are available within the paper.

 \section*{Conflict of interest statement}
  All authors declare no conflicts of interest.

\addcontentsline{toc}{section}{References}

\bibliographystyle{plain} 
{
\bibliography{sample}}
\end{document}